\documentclass[12pt,english]{article}
\usepackage{geometry}
\usepackage{longtable}
\usepackage{booktabs}
\usepackage{multirow}
\usepackage{setspace}
\usepackage[authoryear,round,colon, sort&compress]{natbib}
\bibpunct{(}{)}{,}{autheryear}{,}{,}

\usepackage[english]{babel}
\usepackage{amsmath, amsthm, amssymb, amscd, amsfonts, bm, bbm, mathrsfs}   % math packages
\usepackage{graphicx}
\usepackage{enumerate}
\usepackage[shortlabels]{enumitem}
\usepackage{url} % not crucial - just used below for the URL
\usepackage[colorlinks]{hyperref}
\usepackage{xr-hyper}
\usepackage[capitalize,nameinlink,noabbrev]{cleveref} % to emulate \autoref style
\usepackage[algo2e,linesnumbered,lined,ruled]{algorithm2e}

\usepackage[bottom]{footmisc}
\usepackage[dvipsnames]{xcolor}
\hypersetup{
	breaklinks,
	linkcolor=blue,
	urlcolor=blue,
	anchorcolor=blue,
	citecolor=blue
}

\theoremstyle{plain}

\newtheorem{thm}{Theorem}[section]
\theoremstyle{definition}

\newtheorem*{rem}{Remark}

\theoremstyle{remark}

\RequirePackage{amsthm,amsmath,amsfonts}

\usepackage{multirow}
\usepackage[english]{babel}
\usepackage{amscd, bm, bbm, mathrsfs}   % math packages
\usepackage{graphicx}
\usepackage[shortlabels]{enumitem}
\usepackage[colorlinks]{hyperref}
\usepackage[dvipsnames]{xcolor}
\hypersetup{
	breaklinks,
	linkcolor=blue,
	urlcolor=blue,
	anchorcolor=blue,
	citecolor=blue
}

\allowdisplaybreaks

\newtheorem{lem}{Lemma}[section]
\newtheorem{prop}{Proposition}[section]
\numberwithin{equation}{section}
\numberwithin{figure}{section}
\numberwithin{table}{section}

\newcommand{\bsh}{\backslash} % set difference

\newcommand{\argmin}{\operatornamewithlimits{argmin}} %argmin
\newcommand{\ed}{\overset{\mathsf{d}}{=}} % the same in distribution

\newcommand{\pr}{\mathsf{P}}
\newcommand{\eo}{\mathsf{E}}

\newcommand{\var}{\mathsf{var}}

\newcommand{\nd}{\mathsf{N}}

\newcommand{\ap}{\alpha} %alpha
\newcommand{\g}{\gamma} %gamma
\newcommand{\ga}{\Gamma} %Gamma
\newcommand{\dt}{\delta} % small delta
\newcommand{\Dt}{\Delta} % BIG Delta
\newcommand{\e}{\varepsilon} % epsilon

\newcommand{\ka}{\kappa} % kappa
\newcommand{\s}{\sigma} % sigma
\newcommand{\ld}{\lambda} % small lambda
\newcommand{\B}{\mathbb{B}} % B
\newcommand{\C}{\mathbb{C}} % C
\newcommand{\HH}{\mathbb{H}} % H
\newcommand{\I}{\mathbb{I}} % I
\newcommand{\N}{\mathbb{N}} % N
\newcommand{\R}{\mathbb{R}} % R
\newcommand{\W}{\mathbb{W}} % W
\newcommand{\aaa}{\mathcal{A}}	% A
\newcommand{\bb}{\mathcal{B}}	% B
\newcommand{\dd}{\mathcal{D}}	% D
\newcommand{\ee}{\mathcal{E}}	% E
\newcommand{\jj}{\mathcal{J}}	% J
\newcommand{\pp}{\mathcal{P}}	% P
\newcommand{\rr}{\mathcal{R}}	% R
\newcommand{\sss}{\mathcal{S}}	% S
\newcommand{\ttt}{\mathcal{T}}	% T
\newcommand{\xx}{\mathcal{X}}	% X
\newcommand{\opnorm}[1]{
	{\vert\kern-0.25ex\vert\kern-0.25ex\vert #1 \vert\kern-0.25ex\vert\kern-0.25ex\vert}
}
\newcommand{\opnormbig}[1]{
	{\left\vert\kern-0.25ex\left\vert\kern-0.25ex\left\vert #1 \right\vert\kern-0.25ex\right\vert\kern-0.25ex\right\vert}
}

\newcommand{\opip}[1]{
	{\langle\kern-0.25ex\langle #1 \rangle\kern-0.25ex\rangle}
}
\newcommand{\opipbig}[1]{
	{\left\langle\kern-0.25ex\left\langle #1 \right\rangle\kern-0.25ex\right\rangle}
}

\begin{document}

\newcommand{\thmautorefname}{Theorem}
\newcommand{\defnautorefname}{Definition}
\newcommand{\propautorefname}{Proposition}
\newcommand{\corautorefname}{Corollary}
\newcommand{\lemautorefname}{Lemma}
\newcommand{\remautorefname}{Remark}

\newcommand*{\Appendixautorefname}{Appendix}
\renewcommand{\sectionautorefname}{Section}
\renewcommand{\subsectionautorefname}{Section}
\renewcommand{\subsubsectionautorefname}{Section}

\newcommand{\tred}[1]{{\color{Red}#1}} % textcolor red

\title{Gaussian and bootstrap approximations \\  for functional principal component regression}
\author{
	Hyemin Yeon\thanks{
		Department of Statistics, George Mason University
%		, Email: hyeon2@gmu.edu.
	}
}
\date{}

\maketitle

\begin{abstract}
\noindent
Asymptotic inference using functional principal component regression (FPCR) has long been considered difficult, 
largely because, upon \textit{any scalar scaling}, the FPCR estimator fails to satisfy a central limit theorem, 
leading to the prevailing belief that it is unsuitable for direct statistical inference. 
In this paper, we revisit this traditional viewpoint by establishing a new result: 
upon \textit{suitable operator scaling}, valid Gaussian and bootstrap approximations hold for the FPCR estimator. 
We apply this surprising finding to hypothesis testing for the significance of the slope function in functional regression models and demonstrate the strong numerical performance of the resulting tests. 
Our finding also provides a new framework for projection inference.
While concise, our results yield powerful inferential tools for functional regression.
We believe it paves the way for new lines of inferential methodology for more complex functional regression settings.
\medskip

\end{abstract}

%\tableofcontents

\newpage

\section{Introduction}

\subsection{Background}

As modern scientific technologies advance, 
data are increasingly viewed not as discrete observations but as realizations of underlying functions defined on a continuum. 
Functional data analysis addresses this perspective by treating the data space $\HH$ as infinite-dimensional, 
which introduces substantial inferential challenges. 
Regression, one of the most fundamental frameworks in Statistics and Machine Learning, is no exception, 
and models that incorporate functional predictors or responses fall under the umbrella of \emph{functional regression}. 
This area has been extensively studied for decades in both Statistics and Machine Learning (for a review, see \citealp{Cuevas14, WCM16, Morr15, GRLG24, ST25} among others)
and has found wide-ranging applications across numerous scientific fields \citep{UF13, SGS13, Warm24, OOO25, AD25}.

The most fundamental functional regression model, known as \textit{scalar-on-function regression (SoFR)}, is a linear model in which the response $Y$ remains scalar while the regressor $X$ is a function-valued random element,
linked to $Y$ through the \textit{slope function} $\beta$ in the function space $\HH$.
To mitigate challenges posed by the infinite dimensionality of $\HH$,
a classical and widely used approach is to employ functional principal components (FPCs).
Namely, in the spirit of the finite-dimensional principal component regression \citep[Chapter~8]{Joll86}, 
this method proceeds 
by (i) projecting functional regressors onto the first  $J$ eigenfunctions, 
(ii) performing least squares estimation using these projections as multivariate regressors, 
and (iii) reconstructing the estimated slope function from the fitted coefficients.
The resulting slope estimator is hereafter denoted by $\hat{\beta}_J$ and referred to as the \textit{functional principal component regression (FPCR)} estimator.
It is worth noting that principal components are particularly important in functional regression,
more so than in finite- or even high-dimensional regression, 
because a general Hilbert space does not possess the standard basis as the Euclidean space does. 

In addition to its conceptual generality, 
this methodology enjoys strong asymptotic properties such as minimax optimality in estimation \citep{HH07} and prediction \citep{CH06},
allowing $J \to \infty$ as the sample size diverges. 
While alternative estimators have been proposed in more structured settings, such as those assuming $\beta$ lies in a restricted subset of $\HH$ \cite[e.g.,][]{YC10},
the FPCR estimator nevertheless remains a central estimator used for SoFR
and has been extensively adapted to more general functional regression models,
including those with 
functional \citep{YMW05b, CM13}
or non-Gaussian \citep{DPZ12} responses, 
additional scalar covariates \citep{SL12, KXYZ16}, 
non-linear regression functions \citep{MY08}, 
sparsely observed \citep{ZYZ23} or high-dimensional \citep{CCYZ25} regressor functions, 
and regressor elements taking values in a Riemannian manifold \citep{CP25}.
While our findings can be further extended,
to more clearly deliver our message, 
we focus on the FPCR estimator for SoFR models.

\subsection{Challenges}

Despite its popularity and practical value,
statistical inference based on the FPCR estimator remains challenging,
primarily because its limiting distribution is degenerate.
In a landmark result, \cite{CMS07} showed that, upon \textit{any scalar scaling}, 
the estimator $\hat{\beta}_J$ fails to satisfy a central limit theorem (CLT).
More precisely, for any increasing sequence $\{a_n\}$ of positive real numbers, the asymptotic distribution of $a_n(\hat{\beta}_J-\beta)$ is necessarily degenerate. 
This phenomenon has serious implications for inference;
for instance, it prevents the construction of meaningful confidence regions for $\beta$ \citep{CR18}. 
This difficulty arises,
because the entire function space $\HH$ is too rich, 
yet the FPCR estimator still targets the slope parameter $\beta$ lying in this infinite-dimensional space $\HH$ through pseudo-inversion of the covariance operator.
	Thus, the degeneracy of the limiting distribution of $\hat{\beta}_J$ is closely tied to the ill-posed inverse nature of the problem.
As a consequence, bootstrap developments for $\hat{\beta}_J$ could be meaningless in functional regression models.

This challenge is unique to settings where the regressor lies in an infinite-dimensional space; 
for finite-dimensional linear regression,
CLTs \citep[Example~2.28]{vaart98} and bootstrap theory \citep{free81,Eck18}  were well-established in different ways;
even in high-dimensional regression models, 
where the dimension diverges with the sample size, 
this pathological behavior does not arise \citep{KF00,CL11, DBZ17, CKR23, LS26}.
Consequently, within functional regression, it has long been accepted that the FPCR estimator is unsuitable for direct statistical inference.
Existing work has therefore circumvented this issue 
by (i) focusing on inference for the mean response at a new functional regressor rather than for the slope function itself \citep{CMS07, YDN23RB}, 
(ii) restricting the parameter space to a considerably smaller subset of $\HH$ \citep{SC15,DT24}, 
or (iii) still targeting the full slope function but relying on CLTs for lower-dimensional summaries 
such as quadratic forms \citep{KSM16}, 
the squared norm \citep{KDD22}, 
or extrema of the estimated coefficients \citep{LL25}.
%where the last work borrows state-of-the-art  techniques from high-dimensional statistics .

\subsection{Contributions}

In this work, we overturn this longstanding belief and establish a new  result that enables statistical inference for functional regression: 
upon \textit{suitable operator scaling}, 
Gaussian and bootstrap approximations hold for the FPCR estimator $\hat{\beta}_J$.
Specifically, we show that the distribution of the operator-scaled functional statistic $T_J \equiv \sqrt{n/J} \widehat{\ga}_J^{1/2}(\hat{\beta}_J-\beta)$ can be approximated 
by (i) the Gaussian distribution on $\HH$ with mean zero and covariance equal to a rank-$J$ operator
and (ii) the bootstrap distribution constructed by a residual-type bootstrap method.
Here, $\widehat{\ga}:\HH\to\HH$ denotes the empirical covariance operator constructed by the observed functional regressors, 
while $\widehat{\ga}_J^{1/2}$ is its approximate square-root operator constructed by the first $J$ eigencomponents. 
A key implication is that a meaningful distributional result for $\hat{\beta}_J$ generally cannot be achieved through scalar scaling alone; 
instead, at minimum, one may need scaling by an operator  such as $\widehat{\ga}_J^{1/2}$.
The role of this operator scaling is to scale down the effect of the unbounded pseudo-inverse of the covariance operator involved in estimating $\beta$ by $\hat{\beta}_J$.
This result reveals that a simple yet carefully chosen operator scaling fundamentally changes the asymptotic landscape, 
making principled inference for functional regression genuinely feasible.

As a statistical application, 
we illustrate our finding through hypothesis testing for the significance of the slope function $\beta$.
Namely, we construct the $L^2$ and supremum norms of the functional statistic $T_{0,J} \equiv \sqrt{n/J} \widehat{\ga}_J^{1/2}\hat{\beta}_J$ as test statistics for $H_0:\beta=0$,
with their sampling distributions approximated by the proposed bootstrap method.
Notably,
this supremum-norm statistic has not been previously analyzed in the functional regression literature,
and its approximate sampling distribution is far from straightforward to derive.
In contrast, the $L^2$-based statistic corresponds to a classical quadratic-form statistic \citep{KSM16}, 
and existing bootstrap results \cite[e.g.,][]{KH24} may follow as consequences of our theory that is established under weaker assumptions.
Through numerical studies, we demonstrate that the resulting tests for $H_0:\beta=0$ are competitive with several state-of-the-art testing methods,
showcasing the practical relevance of our finding.

	Our approximation results also lead to an important implication for inference on the projection of $\beta$ along a direction $v \in \HH$.
	In particular, the resulting inference is valid even when the scaling $\tau_J(v)$ defined in \autoref{thmProj} grows faster than $J$, a regime not covered by the classical results of \cite{CMS07, GM11, KH16a}. 
	Thus, under minimal assumptions, the main result of this paper provides a counterpart to the projection-inference results in \cite{YDN23RB}, even when the direction $v$ is non-random.

It is our view that these findings will have substantial impact on the regression literature involving infinite‑dimensional data and infinite-dimensional parameters,
as many existing approaches for SoFR models have been successfully extended to far more general settings. 
For example, methods developed for SoFR with a single functional regressor naturally extend to
SoFR with multivariate functional regressors \citep[Section~4]{CGLC24},
generalized SoFR models \citep{DPZ12, SC15},
non-linear functional regression \citep{MY08, PCTM18},
SoFR with extra scalar covariates \citep{KXYZ16, LZ20}, 
models involving functional responses \citep{KMSZ08, CM13}, 
and even regression frameworks with manifold data \citep{CP25}.
Consequently, although our theoretical and methodological advancements are rooted in the simple SoFR setting, 
they are poised to influence a broad class of functional regression models and to support future advances across this expanding research area.

Before concluding the Introduction, we elaborate on several technical aspects of our contribution.
First, the main theoretical results we obtain are  genuinely new to the literature;
although the statements are simple, their implications are substantial. 
Second, the technical arguments required for our result do not follow directly from existing work such as \cite{CMS07} and \cite{YDN23RB}. 
While we draw on certain auxiliary lemmas from these papers, 
the underlying problems differ in essential ways, 
necessitating new analytical tools. 
Third, 
our bootstrap consistency result deserves separate emphasis:
unlike many bootstrap results in Statistics, 
the bootstrap consistency result does not rely on \textit{any} Gaussian approximation. 
Remarkably, the assumptions required for bootstrap consistency are weaker than those imposed for the Gaussian approximation itself.
In fact, in the development of this paper, the Gaussian and bootstrap results were obtained independently;
the bootstrap consistency was established first, and only subsequently did it become clear that a corresponding Gaussian approximation could be derived.
Nevertheless, the Gaussian approximation remains crucial for applying the bootstrap to statistical inference, since the continuity and boundedness of the resulting test statistics’ densities help ensure the validity of their statistical guarantees.
Fourth,
establishing the Gaussian approximation is somewhat more delicate, 
requiring 
%a combination of functional calculus techniques \cite[e.g.,][]{CMS07} and 
bounds on the Wasserstein distance to multivariate normal distributions \cite[e.g.,][]{Bonis20}.
Finally, our main approximation results are formulated in terms of the Wasserstein distance, which yields a stronger mode of distributional approximation; for example, they imply approximations in the Prohorov metric.

\subsection{Organization of the paper}

This paper is structured as follows.
In \autoref{sec_main}, after reviewing the FPCR estimation and a subsequent residual bootstrap, 
we present our main approximation results.
	\autoref{secStatApp} then applies these results to hypothesis testing for the significance of the slope function and projection inference onto a non-random direction $v \in \HH$, and examine their performance numerically.
A brief concluding remark is given in \autoref{sec_conc}, and all technical details are deferred to the Appendix.

%\tred{
	%1) importance/prevalence of CLT in classical statistics, particularly in multivariate/high-dimensional cases.
	%2) but CLT may not hold for functional regression upon any \textit{scalar scaling},
	%which is not even in high-dim
	%3) we found that \textit{upon operator scaling} Gaussian approximation hold. This operator scaling looks optimal, as without it, bias term not vanishes \cite[Proposition~1?]{CMS07}
	%4) more on this?
	%5) application to hypothesis testing
	%6) organization
	%}

\section{Main results} \label{sec_main}

%Upon introducing the FPCR estimator and a residual bootstrap for SoFR  through \autoref{
	%
	%\tred{separate these into three:
		%	1) notation for FLRM and conditions
		%	2) Gau approx upon Prohorov and FPCR estimation
		%	3) BTS approx upon Wasserstein and residual bootstrap }
	%
	%\tred{ need to edit intro and proof all parts in the paper accordingly}
	
	\subsection{SoFR models}

	We start with setting our theoretical framework. 
	Let $\HH$ denote the underlying function space,
	which we assume to be an infinite-dimensional separable Hilbert space.
	Main examples include the space 
	$L^2([0,1]) \equiv \left\{ f:[0,1]\to\R:\int_0^1 f(u)^2 du < \infty \right\}$ 
	of square-integrable real-valued functions on $[0,1]$ with inner product $\langle f,g \rangle \equiv \int_0^1 f(u)g(u)du$
	and the Sobolev space 
	$\W^{1,2}([0,1]) \equiv \left\{ f \in L^2([0,1]): f' \in L^2([0,1]) \right\}$ 
	of (once differentiable) functions on $[0,1]$ whose derivatives are also square-integrable, with inner product $\langle f,g \rangle \equiv \int_0^1 f(u)g(u) du + \int_0^1 f'(u)g'(u) du$. 
	%Another notable example is the Bayes space of density functions \citep{VEP14}.
	We consider the following linear model with scalar response $Y$ but functional regressor $X$:
	\begin{align} \label{eq_sofr}
		Y = \ap + \langle \beta, X \rangle + \e.
	\end{align}
	Here, $\beta \in \HH$ is the slope parameter of importance in our framework;
	$X$ is the random element that takes values in $\HH$;
	$\e$ denotes the scalar-valued random error with homoscedastic variance $\s^2 \equiv \eo[\e^2|X] \in (0,\infty)$; and
	$\ap \in \R$ is the scalar intercept parameter.

	For $x,y \in \HH$, we define the tensor product $x\otimes y:\HH\to\HH$ by $(x\otimes y)(z) = \langle z,x \rangle y$ for $z \in \HH$. 
	Throughout the paper, we assume that $X$ has finite second moment as $\eo[\|X\|^2]<\infty$.
	The covariance operator $\ga$ of $X$ is defined by $\ga \equiv \eo[(X-\eo[X])^{\otimes 2}]$, 
	where $x^{\otimes 2} \equiv x \otimes x$ is the shorthand for the squared tensor product of $x \in \HH$. 
	Since the covariance operator $\ga$ is a positive semi-definite, self-adjoint, and compact operator on $\HH$,
	and hence admits spectral decomposition as 
	$\ga =\sum_{j=1}^\infty \g_j \phi_j^{\otimes 2}$,
	where $(\g_j,\phi_j)$ is the $j$-th eigenvalue-eigenfunction pair of $\ga$.
	Here, the eigenvalues $\{\g_j\}_{j=1}^\infty$ satisfy $\g_1 \geq \g_2 \geq \cdots \geq 0$ and $\g_j\to0$ as $j\to\infty$,
	while the eigenfunctions $\{\phi_j\}_{j=1}^\infty$  form an orthonormal system of $\HH$. 
	For later developments, 
	we consider the sequence $\{\dt_j\}_{j=1}^\infty$ of eigengaps defined as $\dt_1 \equiv \g_1 - \g_2$ and $\dt_j \equiv \min\{ \g_j - \g_{j+1}, \g_{j-1} - \g_j \}$ for integer $j \geq 2$.
	We guide the readers to \cite{HE15} for general Hilbert space theory related to functional data analysis. 
	
	We consider the following systematic conditions:
	\begin{enumerate}[label=(A\arabic*)]
		\item $\ker \ga \equiv \{ x \in \HH: \ga x = 0 \} = \{0\}$; and \label{condA1ker}
		\item the eigenvalues $\{\g_j\}_{j=1}^\infty$ are positive and distinct. \label{condA2ev}
	\end{enumerate}
	Condition~\ref{condA1ker} guarantees the identifiability of the SoFR model in \eqref{eq_sofr} \cite[e.g.,][]{CFS99}.
	Through Condition~\ref{condA2ev}, 
	we can simplify the framework by removing potential multiplicities of the eigencomponents 
	and exclude some trivial cases where the functional regressors are supported on a finite-dimenionsional subspace \cite[e.g.,][]{HH07}. 
	Unless otherwise stated, the subsequent discussion proceeds under the above framework together with Conditions~\ref{condA1ker}--\ref{condA2ev}.

	\subsection{FPCR estimator and residual bootstrap}
	
	To describe the FPCR estimation, 
	we suppose that the data observations $\{(Y_i, X_i)\}_{i=1}^n$ are independently generated from the SoFR model \eqref{eq_sofr}. 
	Then, the sample covariance operator $\widehat{\ga}$ of the regressors $\{X_i\}_{i=1}^n$ is defined by $\widehat{\ga} \equiv n^{-1} \sum_{i=1}^n (X_i - \bar{X})^{\otimes 2}$, where $\bar{X} \equiv n^{-1} \sum_{i=1}^n X_i$. 
	It also admits spectral decomposition as 
	$\widehat{\ga} =\sum_{j=1}^n \hat{\g}_j \hat{\phi}_j^{\otimes 2}$,
	where $\hat{\g}_j$ and $\hat{\phi}_j$ are the $j$-th eigenvalue and eigenfunction of $\widehat{\ga}$, respectively.
	The FPCR estimator is then obtained through the following steps.
	First, we project the functional regressors $\{X_i\}_{i=1}^n$ onto the first $J$ eigenfunctions $\{\hat{\phi}_j\}_{j=1}^J$,
	while allowing $J\to\infty$ as $n\to\infty$.
	Second, denoting by $\bm{X}_i \equiv [\langle X_i, \hat{\phi}_j \rangle]_{1 \leq j \leq J}$ the $i$-th projected regressor onto the first $J = J_n$ eigenfunctions for each $i=1,\dots,n$, 
	we find the least square estimator 
	$\hat{\bm{\beta}}_{\mathrm{coef},J} 
	= (\hat{\beta}_{\mathrm{coef},1}, \dots, \hat{\beta}_{\mathrm{coef},J})^\top$
	from the multiple linear regression of the original responses $\{Y_i\}_{i=1}^n$ on $\{\bm{X}_i\}_{i=1}^n$.
	Finally, the FPCR estimator $\hat{\beta}_J$ is defined by
	\begin{align} \label{eq_fpcr_est}
		\hat{\beta}_J \equiv \sum_{j=1}^J \hat{\beta}_{\mathrm{coef},j} \hat{\phi}_j.
	\end{align}

	We introduce a residual bootstrap procedure to create bootstrap data $\{(Y_i^*, X_i)\}_{i=1}^n$ to mimic the original observations $\{(Y_i, X_i)\}_{i=1}^n$. 
	Considering the (centered) residuals 
	\begin{align*}
		\hat{\e}_i  \equiv Y_i - \bar{Y} - \langle \hat{\beta}_J, X_i - \bar{X} \rangle, \quad i=1,\dots,n,
	\end{align*}
	we draw bootstrap errors $\e_1^*, \dots, \e_n^*$ uniformly from $\{\hat{\e}_i\}_{i=1}^n$. 
	Then, following the original SoFR model in \eqref{eq_sofr}, the bootstrap responses are computed as
	\begin{align*}
		Y_i^* \equiv \bar{Y} + \langle \hat{\beta}_J, X_i - \bar{X} \rangle + \e_i^*, \quad i=1,\dots,n.
	\end{align*}
	%Here, $J$ is an additional truncation level, where the corresponding FPCR estimator $\hat{\beta}_J$ is used for computing residuals and generating the bootstrap responses but not for  main bootstrap inference tasks. 
	We then construct a bootstrap version $\hat{\beta}_J^*$ of the FPCR estimator akin to the original estimator $\hat{\beta}_J$  in \eqref{eq_fpcr_est}. 
	Specifically, 
	using the fitted coefficients $\hat{\bm{\beta}}_{\mathrm{coef},J}^* = (\hat{\beta}_{\mathrm{coef},1}^*, \dots, \hat{\beta}_{\mathrm{coef},J}^*)^\top$ from the linear regression of the bootstrap responses $\{Y_i^*\}_{i=1}^n$ on the projected regressors $\{\bm{X}_i\}_{i=1}^n$,
	the bootstrap FPCR estimator $\hat{\beta}_J^*$ is then given by
	\begin{align} \label{eq_fpcr_est_bts}
		\hat{\beta}_J^* \equiv \sum_{j=1}^J \hat{\beta}_{\mathrm{coef},j}^* \hat{\phi}_j. 
	\end{align}

	\subsection{Gaussian and bootstrap approximation results}

	\newcommand{\wsd}{\mathsf{W}}
	Our approximation results are stated with the (second-order) Wasserstein distance between two probability distributions. 
	Let $\B$ be a generic Banach space with norm $\|\cdot\|$ and we write $\pp(\B)$ for the set of all probability distributions on $\B$.
	Then, the Wasserstein metric $\wsd$ on $\pp(\B)$ is defined by 
	\begin{align} \label{eqWassDist}
		\wsd(P,Q) \equiv \inf_{U' \sim P, V' \sim Q} \eo[\|U'-V'\|^2]^{1/2}, \quad P\in \pp(\B), \quad Q \in \pp(\B).
	\end{align}
	With a slight abuse of notation, 
	for $P\in \pp(\B)$ and $Q \in \pp(\B)$, and two random elements $U \sim P$ and $V \sim Q$ taking values in $\B$,
	we write 
	\begin{align*}
		\wsd(U,V) \equiv \wsd(P, Q) = \inf_{U' \ed U, V' \ed V} \eo[\|U'-V'\|^2]^{1/2}
	\end{align*}
	and call it the Wasserstein distance between $U$ and $V$.
	Since our results hold conditionally on the observed regressors $\xx_n \equiv \{X_i\}_{i=1}^n$, we work with the conditional distribution $\pr^X \equiv \pr(\cdot|\xx_n)$ given $\xx_n$,
	leading to defining the conditional Wasserstein distance $\wsd^X$ by replacing $\eo$ with the conditional expectation $\eo^X \equiv \eo[\cdot|\xx_n]$ in \eqref{eqWassDist},
	i.e.,
	\begin{align} \label{eqWassDistCond}
		\wsd^X(P,Q) \equiv \inf_{U' \sim P, V' \sim Q} \eo^X[\|U'-V'\|^2]^{1/2}, \quad P\in \pp(\B), \quad Q \in \pp(\B).
	\end{align}
	We note that the probability distributions $P,Q$ in \eqref{eqWassDistCond} can depend on the regressors $\xx_n \equiv \{X_i\}_{i=1}^n$.
	We refer to \citet[Section~8]{Mallows72} for more details about the Wasserstein distance and its theoretical properties.

	We list technical conditions for our approximation results. 
	In what follows, all asymptotic statements are understood to hold as the sample size $n$ diverges to infinity.
	\begin{enumerate}[label=(C\arabic*)]
		\item $\sup_{j \in \N} \g_j^{-2} \eo[\langle X-\eo[X]	, \phi_j \rangle^4] < \infty$; \label{condMomentX}
		
		\item $\g_j$ is a convex positive function of $j$ for sufficiently large $j$; \label{condConvexEV}
		
		\item $\sup_{j \in \N} \g_j j \log j  < \infty$; \label{condDecayEV}
		
		\item  \label{condBasicTrunc}
		$n^{-1} \sum_{j=1}^J \dt_j^{-2} + n^{-1/2} J^2 \log J = o(1)$;
		
	\end{enumerate}
	%	\begin{enumerate}[label=(B$_v$)]
		%		\item \label{condSlope}
		%		$J^v \sum_{j>J} \langle \beta, \phi_j \rangle^2 = o(1)$ for $v \in (1,\infty)$; and
		%		
		%	\end{enumerate}
	%	\begin{enumerate}[label=(E)]
		%		\item  \label{condErr}$\eo[\e^4|X] \equiv \eo[\e^4] < \infty$; 
		%	\end{enumerate}
	%	\begin{enumerate}[label=(R$_w$)]
		%		\item  \label{condTrunc} $J = o(n^{1/w})$ for $w \in (0,\infty)$.
		%		
		%	\end{enumerate}
	%	Conditions~\ref{condMomentX}--\ref{condTrunc} are required to apply theory of functional calculus, which is a standard technique in functional data analysis; see \cite{CMS07} and \cite{YDN23RB} for example.
	%	Condition~\ref{condMomentX} ensures that the normalized functional principal component scores have finite fourth moments,
	%	implying that $\eo[\|X\|^4]<\infty$.
	%	The next tree conditions, \ref{condConvexEV}--\ref{condBasicTrunc}, are basic assumptions that enable the derivation of perturbation lemmas (cf.~\autoref{appTechBias}). 
	%	Condition~\ref{condSlope} implies a mild smoothness requirement on the slope function $\beta$,
	%	which holds, for example, when $|\langle \beta, \phi_j \rangle| \asymp j^{-b}$ with $b > (1+v)/2$.
	%	The finite fourth moment condition \ref{condErr} on the error is used to derive the bootstrap approximation of the variance term (cf.~\autoref{appTechVar}).
	%	Finally, Condition~\ref{condTrunc} specifies a sufficient growing rate for the truncation level $J$,
	%	where a larger value of $w$ corresponds to a slower growth rate.
	%	%for instance, it is satisfied when $J \asymp n^{1/(u+\ka)}$ for some $\ka>0$.
	Conditions~\ref{condMomentX}--\ref{condBasicTrunc} are required to apply theory of functional calculus, which is a standard technique in functional data analysis; 
	see \cite{CMS07} and \cite{YDN23RB} for example.
	Condition~\ref{condMomentX} ensures that the normalized FPC scores have finite fourth moments,
	implying that $\eo[\|X\|^4]<\infty$.
	The next three conditions, \ref{condConvexEV}--\ref{condBasicTrunc}, facilitate the derivation of perturbation lemmas (cf.~\autoref{appTechBias}). 
	Among them, Condition~\ref{condDecayEV} is implicitly assumed in exiting literature, such as in \cite{CMS07}, 
		despite a pathological case where it fails to hold.
		For completeness, we provide this counterexample in \autoref{egCounterCond3} of the Appendix.
		Condition~\ref{condBasicTrunc} regulates the divergence of $J$ through two components. 
		The first part $n^{-1}\sum_{i=1}^n \dt_j^{-2} = o(1)$ enables us to replace the sample contours with their population versions for functional calculus; 
		see \autoref{appTechBias} and \citet[Lemma~S2]{YDN23RB} for a related discussion.
		Under polynomial decay rate $\dt_j \asymp j^{-a-1}$ for $a>1$, this can be translated as $J= o(n^{1/(2a+3)})$,
		where $r_j \asymp s_j$ means that the ratio $r_j/s_j$ of two sequences $\{r_j\}$ and $\{s_j\}$ of positive real numbers is bounded away from zero and infinity.
		The second part $n^{-1/2} J^2 \log J = o(1)$ is required to determine moment bounds via Condition~\ref{condConvexEV} and the convexity lemma (\autoref{lemCMS07}(a) of the Appendix). 
		Such mild conditions are ubiquitous in functional data analysis \citep{HH07,LL25} even within non-FPC-based methodologies
		\citep{SC15}.

	We now provide our new result
	for estimating the sampling distribution of the operator-scaled functional statistic
	\begin{align} \label{eqStat}
		T_J \equiv \sqrt{n / J} \widehat{\ga}_J^{1/2}(\hat{\beta}_J - \beta)
	\end{align}
	%with the standardized Gaussian distribution supported on $\HH_J \equiv \mathrm{span}(\{\phi_j\}_{j=1}^J)$ or 
	with the bootstrap distribution of its bootstrap version
	\begin{align} \label{eqStatBTS}
		T_J^* \equiv \sqrt{n / J} \widehat{\ga}_J^{1/2}(\hat{\beta}_J^* - \hat{\beta}_J)
	\end{align}
	conditionally on the regressors $\xx_n \equiv \{X_i\}_{i=1}^n$.
	Here, the operator scaling $\widehat{\ga}_J^{1/2} \equiv \sum_{j=1}^J \hat{\g}_j^{1/2} \hat{\phi}_j^{\otimes 2}$ is the pseudo-square-root of the sample covariance operator $\widehat{\ga}$.
	The proof of the following \autoref{thmMain} is given in Appendix~\ref{appTechMain1}.

		\begin{thm} \label{thmMain}
			
			Suppose that Conditions~\ref{condMomentX}--\ref{condBasicTrunc} hold.
			\begin{enumerate}[(a)]

				\item (Gaussian approximation)
				We further suppose that $\eo[\e^4|X] \equiv \eo[\e^4] < \infty$ and $J = o(n^{1/w})$ for some $w >6$. 
				Then, the sampling distribution of $T_J$ is approximated by a Gaussian distribution as
				\begin{align*}
					%				\wsd^X ( T_J, \s J^{-1/2}G_J ) = o_\pr(1),
					\wsd^X ( T_J, \s J^{-1/2}G_J ) = O_\pr(n^{-1/2}J),
				\end{align*}
				where $G_J$ is a Gaussian random element taking values in $\HH$ with mean zero and covariance operator $\widehat{\Pi}_J \equiv \sum_{j=1}^J \hat{\phi}_j^{\otimes 2}$.

				\item (Bootstrap approximation)
				Without any further assumptions beyond Conditions~\ref{condMomentX}--\ref{condBasicTrunc}, 
				the residual bootstrap is valid for approximating the distribution of $T_J$ as
				\begin{align*}
					%				\wsd^X ( T_J, T_J^* )=o_\pr(1).
					\wsd^X ( T_J, T_J^* )=O_\pr(n^{-1/2}) + O_\pr(J/n) + O \left( \sum_{j>J}\g_j \langle \beta, \phi_j \rangle^2 \right).
				\end{align*}

			\end{enumerate}
			
		\end{thm}

		\autoref{thmMain} have several implications.
		First, the required conditions are minimal; 
		bootstrap does not impose any further conditions outside Conditions~\ref{condMomentX}--\ref{condBasicTrunc}.
		In particular, both approximation results does not require any condition for the true slope parameter $\beta$,
		which is contrasted with existing results \cite[e.g.,][]{HH07, CH06,CMS07,YDN23RB,YC10}. 
		Second, the bootstrap result requires less conditions than the Gaussian approximation.
		Nevertheless, the additional conditions for the Gaussian approximation are not stringent. 
		The finite fourth moment in $\eo[\e^4|X] \equiv \eo[\e^4] < \infty$ is assumed to use multivariate Gaussian approximation in \cite{Bonis20};
		the assumption that $J = o(n^{1/w})$ for some $w >6$ is standard when using perturbation theory \cite[e.g.,][]{CMS07,YDN23RB},
		where  more details can be found in \autoref{lemPerturb2} of \autoref{appTechBias}.
		Third, the functional statistic $T_J$ in \eqref{eqStat} does not exhibit a \textit{limiting} distribution, instead it is \textit{approximated} by a Gaussian distribution with identity-like operator $\widehat{\Pi}_J$.
		Approximation by a Gaussian distribution is still crucial,
		as it helps finding bootstrap approximation of the scalar-valued statistics in the following sections with respect to Kolmogorov distance.

	\begin{rem}
		In both statistics \eqref{eqStat}--\eqref{eqStatBTS}, 
		the scaling terms $J^{-1/2}$ and $\widehat{\ga}_J^{1/2}$ are both essentially indispensable alongside the usual factor $\sqrt{n}$.
		Specifically,
		without scaling by $\widehat{\ga}_J^{1/2}$, 
		the variance terms could fail to stabilize,
		%		 (cf.~Propositions~\ref{propGauVar} and~\ref{propVar} in the Appendix),
		and omitting either $J^{-1/2}$ or $\widehat{\ga}_J^{1/2}$ prevents the bias term from becoming negligible.
		%		(cf.~\autoref{lemPerturb1} in the Appendix). 
		%		\tred{} further, this operator scaling removes bias in the original statistic, which is rare in the literature and useful for deriving this result.
	\end{rem}

	\section{Statistical applications} \label{secStatApp}

		Our main theoretical development naturally produces approaches for two important statistical tasks in the SoFR literature: 
		inference for the entire slope $\beta$ and its projection onto a direction $v \in \HH$.

	\subsection{Hypothesis testing for the significance of  the slope function} \label{ssecTest}
	
	%$X$ mean square continuous, \cite[Section~7.3]{HE15} $\lim_{m\to\infty}\eo[\|X(t_m)-X(t)\|^2]$ for any $t \in [0,1]$ and sequence $\{t_m\}_{m=1}^\infty$ in $[0,1]$ converging to $t$. 

	Hypothesis testing for the slope parameter is one of the most important problems for functional regression, analogously to those for finite-dimensional regression.
	The null hypothesis is represented as $H_0:\beta=0$. 
	Historically, quadratic-form statistics have been widely used, 
	as they asymptotically resemble the classical F test in finite-dimensional cases \citep{CFMS03, HMV13, lei14, KSM16}.
	Recently, \cite{LL25} develop a new test using extrema statistics and along with state-of-the-art techniques including high-dimensional bootstrap and partial standardization \cite[cf.][]{LLM23}. 
	
	In this section, we apply our approximation results to this fundamental problem.
	As test statistics, we consider the $L^2$ and supremum norms of the functional test statistic $T_{0,J} \equiv \sqrt{n/J} \widehat{\ga}_J^{1/2} \hat{\beta}_J$.
	These statistics and their bootstrap counterparts are defined by
	\begin{alignat}{3}
		S_{\mathrm{sq},J}
		& \equiv  \int_0^1 T_{0,J}(u)^2 du,
		&& \qquad S_{\mathrm{sup},J}
		&& \equiv \sup_{u \in [0,1]} |T_{0,J}(u)|, \label{eqStatTest} \\
		S_{\mathrm{sq},J}^*
		& \equiv  \int_0^1 T_J^*(u)^2 du,
		&& \qquad S_{\mathrm{sup},J}^*
		&& \equiv \sup_{u \in [0,1]} |T_J^*(u)|. \label{eqStatTestBTS}
	\end{alignat}
	%\begin{align}
	%	S_{\mathrm{sq},J}
	%	& \equiv \int_0^1 T_{0,J}(u)^2 du, \label{eqStatL2} \\
	%	S_{\mathrm{sup},J}
	%	& \equiv \sup_{u \in [0,1]} |T_{0,J}(u)|. \label{eqStatSup}
	%\end{align}
	%Their bootstrap counterparts are given by  
	%\begin{align}
	%	S_{\mathrm{sq},J}^*
	%	& \equiv \int_0^1 T_J^*(u)^2 du, \label{eqStatL2} \\
	%	S_{\mathrm{sup},J}
	%	& \equiv \sup_{u \in [0,1]} |T_J^*(u)|. \label{eqStatSup}
	%\end{align}
	%where $T_{0,J}^* \equiv \sqrt{n/J} \widehat{\ga}_J^{1/2} \hat{\beta}_J^*$.
	Based on the following theorem, for each $l \in \{\mathrm{sq},\mathrm{sup}\}$, 
	in the Kolmogorov distance,
	the distributions of the original statistics $S_{l,J}$ in \eqref{eqStatTest} are consistently estimated by their bootstrap counterparts of $S_{l,J}^*$ in \eqref{eqStatTestBTS}, 
	which helps calibrating the critical value for testing. 
	Its proof is deferred to Appendix~\ref{appTechTest}.

		\begin{thm} \label{thmTest}
			Suppose that the assumptions of \autoref{thmMain}(a) hold.
			%		Suppose that Conditions~\ref{condMomentX}--\ref{condBasicTrunc} and \ref{condErr} hold along with Conditions~\ref{condSlope} and~\ref{condTrunc}.
			Under $H_0:\beta=0$, we have 
			\begin{align*}
				\sup_{s \in \R} |\pr(S_{l,J} \leq s|\xx_n) - \pr^*(S_{l,J}^* \leq s|\xx_n)| = o_\pr(1)
			\end{align*}
			(i) for $l=\mathrm{sq}$ when $\HH = L^2([0,1])$ 
			%		if $v>2$ and $w>6$
			and (ii) for $l=\mathrm{sup}$ when $\HH = \W^{1,2}([0,1])$ 
			%		if $v>2.25$ and $w>6.5$. 
			if it additionally holds that $J^{3/4} \sum_{j>J} \g_j \langle \beta, \phi_j \rangle^2 = o(1)$.
		\end{thm}

	We made several observations about these test statistics.
	First, the supremum statistic $S_{\mathrm{sup},J}$ is analyzed for the first time in our work.
	Although its validity for testing $H_0:\beta=0$ does not follow directly from classical theory,
	it emerges immediately from our approximation result in \autoref{thmMain} 
	together with (i) continuity of the supremum norm mapping on $\W^{1,2}([0,1])$ ensured by an embedding theorem 
	and (ii) the continuity and boundedness of the density of the absolute supremum of a Gaussian process; 
	see Lemmas~\ref{lemSobolevEmb} and~\ref{lemSupDensity} of the Appendix for these observations, respectively.
	The extra condition $J^{3/4} \sum_{j>J} \g_j \langle \beta, \phi_j \rangle^2 = o(1)$ is mild, which covers most decay rates;
		for instance, if $\g_j \asymp j^{-a}$ and $|\langle \beta, \phi_j \rangle| \asymp j^{-b}$ for some $a>1$ and $b>1/2$, then it is satisfied when $a+2b>7/4$.
	On the other hand, the $L^2$ statistic $S_{\mathrm{sq,J}}$ is essentially equivalent to those used by \citet{CFMS03, HMV13, lei14, KSM16}.
	Its sampling distribution is shown to be approximated with a (scaled) chi-square distribution with $J$ degrees of freedom, 
	as established by prior work such as \cite{KSM16} and \cite{KH24} under stringent assumptions,
	where the bootstrap result in the latter work can be a direct consequence of \autoref{thmMain}.
	More generally, \autoref{thmMain} combined with \autoref{lemWSDcont} of the Appendix implies that a bootstrap approximation under the Wasserstein distance hold for any Lipschitz continuous function of $T_J$.

	To facilitate numerical comparison with existing test procedures, we adopt the simulation settings of \cite{LL25}, 
	enabling us to compare our test statistics 
	with their proposed test and 
	with the methods by \cite{HMV13, lei14, SC15} that \cite{LL25} have been already assessed.
	Among these, only the test by \cite{SC15} does not rely on FPCs.
	We call these four tests the benchmark methods in what follows.
	We generate random samples $\{(Y_i,X_i)\}_{i=1}^n$ of size $n \in \{50, 200\}$ from the SoFR model \eqref{eq_sofr}, imposing $\eo[X]=0$ and $\eo[Y]=0$, hence $\ap=0$, for simplicity. 
	The functional regressors $\{X_i\}_{i=1}^n$ are independent centered Gaussian processes with the Mat\'{e}rn covariance function
	\begin{align*}
		C(u_1,u_2)
		= C(u_1,u_2; \ka^2, \rho_{\mathrm{Mat}}, \nu) 
		\equiv \ka^2 {2^{1-\nu} \over \ga_{\mathrm{func}}(\nu)} \left( {\sqrt{2\nu} |u_1-u_2| \over \rho_{\mathrm{Mat}}} \right)^\nu K_\nu \left( {\sqrt{2\nu} |u_1-u_2| \over \rho_{\mathrm{Mat}}} \right)
	\end{align*}
	for $(u_1,u_2)^\top \in [0,1]^2$,
	where $\ga_{\mathrm{func}}$ denotes the gamma function, $K_\nu$ is the modified Bessel function of the second kind, and the corresponding covariance operator is given by the integral operator as $(\ga x)(u) = \int_0^1 C(u,u') x(u') du'$ for $u \in [0,1]$ and $x \in L^2([0,1])$.
	We fix $(\ka^2, \rho_{\mathrm{Mat}}) = (1,1)$ but vary $\nu \in \{1/2,1,3/2\}$ to see the effects of the smoothness of $\{X_i\}_{i=1}^n$. 
	%	$\nu=2$ (smoother), $\nu=1/4$ (wiggler). 
	%very smooth: 1st eigencomponent has 83.61\% explanation
	Following \cite{LL25}, we generate the errors from the Laplacian distribution with zero mean and unit variance. 
	We use the set $\{\phi_{\mathrm{tri},j}\}_{j=1}^\infty$ of trigonometric functions for constructing slope functions:
	\begin{align}
		\phi_{\mathrm{tri},1}(u) = 1, \ \ 
		\phi_{\mathrm{tri},2m}(u) = \sqrt{2} \sin (2m\pi u), \ \ 
		\phi_{\mathrm{tri},2m+1}(u) = \sqrt{2} \cos (2m\pi u), \label{eqBasisTri}
	\end{align}
	for $u \in [0,1]$ and $m \in \N$.
	The true slope parameter is $\beta = c \beta_1$, where the degree $c$ of alternatives range over $c \in \{0, 0.2, 0.4, 0.6, 0.8, 1\}$ and the alternative function $\beta_1$ is considered among the following:
	\begin{itemize}
		\item (sparsest) $\beta_1(u) = 1$;
		
		\item (sparse) $\beta_1(u) = \sum_{j=1}^3 {11 \over 4} (j+2)^{-1} \phi_{\mathrm{tri},j}(u)$;
		
		\item (dense) $\beta_1(u) = \sum_{j=1}^{100} {11 \over 4} (j+2)^{-1} \phi_{\mathrm{tri},j}(u)$; and
		
		\item (densest) $\beta_1(u) = {6 \over 4} u^2 e^u$.
		
	\end{itemize}
		As suggested by a reviewer, we have expanded simulation design by adding extra scenarios as follows:
		\begin{itemize}
			\item (smoother) $\beta_1(u) = \sum_{j=1}^{100} 20(2+j)^{-3} \phi_{\mathrm{tri},j}(u)$;
			\item (smoothest) $\beta_1(u) = \sum_{j=1}^{100} 200(2+j)^{-5} \phi_{\mathrm{tri},j}(u)$; and
			\item (exponential) $\beta_1(u) = \sum_{j=1}^{100} 3e^{-j} \phi_{\mathrm{tri},j}(u)$.
	\end{itemize}

	All functions are evaluated at 200 equi-distant time grid points on $[0,1]$,
	the tests are conducted at significance level 0.05,
	the experiments are repeated with 1,000 iterations,
	and the number of bootstrap resamples are set to be 1,000.
	We record whether the bootstrap tests based on the statistics in \eqref{eqStatTest} reject the null $H_0:\beta=0$ and compute the averages of the rejection indicators to present empirical rejection rates.

		The Gaussian approximations could also be used to construct tests after estimating the error variance $\s^2$. 
		For instance, one may use $\hat{\s}_J^2 \equiv n^{-1} \sum_{i=1}^n (Y_i - \bar{Y} - \langle \hat{\beta}_J, X_i - \bar{X} \rangle)^2$, 
		which is consistent for $\s^2$ whenever $\hat{\beta}_J$ is consistent for $\beta$ \cite[cf.][Corollary~S1]{YDN23RB}. 
		However, as observed in \cite{YDN23RB} for mean-response inference,
		inference based directly on such Gaussian approximations are expected to perform poorly in finite samples. 
		For this reason, we do not report these Gaussian-approximation tests.

		The truncation level $J$ is selected by the minimum among the candidates $\jj \equiv \{1, \dots, 49\}$ whose fractions of variance explained exceed a certain threshold $\rho \in (0,1)$ \cite[cf.][Section~12.2]{KR17}.
		Specifically, 
		we select the truncation level as
		\begin{align}
			\widehat{J} 
			= \widehat{J} (\rho)
			\equiv \argmin_{J \in \jj} \left\{ {\sum_{j=1}^J \hat{\g}_j \over \sum_{j=1}^{n-1} \hat{\g}_j} \geq \rho \right\}. \label{eqFVE}
		\end{align}
		We report the results when picking $\rho=0.75$, as the subsequent tests perform well in general.
		See \autoref{sec_conc} for more discussion on the choice of the truncation level for FPCR estimation.

		\begin{figure}[b!]
			\centering
			\includegraphics[width=0.89\linewidth]{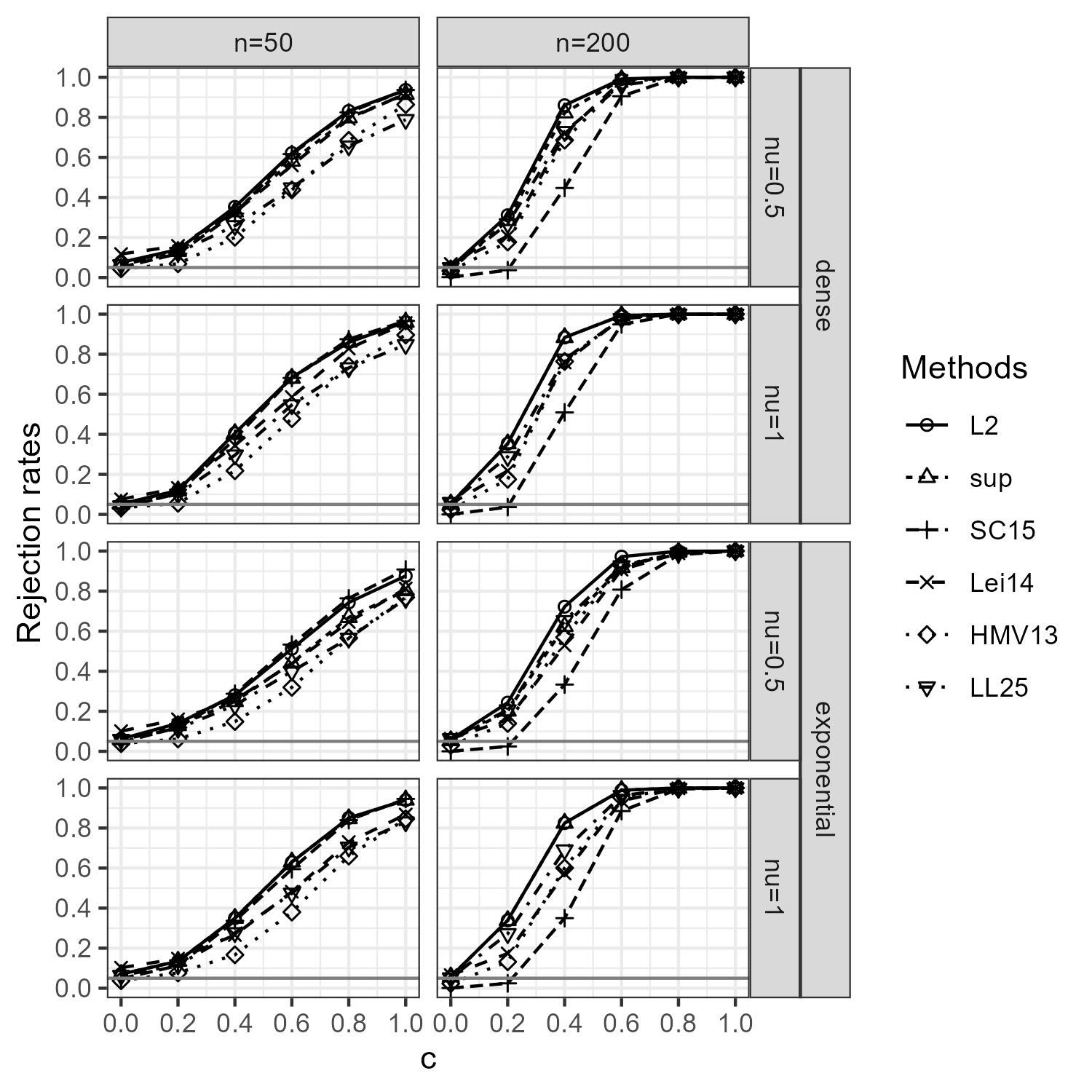}
			\caption{
				Empirical rejection rates of the bootstrap tests using statistics  $S_{\mathrm{sq},J}$ and $S_{\mathrm{sup},J}$ in \eqref{eqStatTest} and the benchmark tests.
			}
			\label{fig_test1sofr}
		\end{figure}

		\autoref{fig_test1sofr} presents the power curves for the tests based on \(S_{\mathrm{sq},J}\) and \(S_{\mathrm{sup},J}\) in \eqref{eqStatTest}, 
		with critical values calibrated via the bootstrap approximations in \autoref{thmTest}. 
		While only selected scenarios are displayed here, the full results are provided in \autoref{fig_test1sofr_appendix}. 
		Notably, unlike the conservative sizes of certain benchmark methods, 
		our bootstrap tests maintain empirical sizes much closer to the nominal 0.05 level. 
		This confirms that the approximations in Theorems~\ref{thmMain}--\ref{thmTest} are highly accurate. 
		Furthermore, both proposed tests display substantial power even under weak signals (\(c=0.2\) or \(c=0.4\)). 
		%	particularly when the sample size is small as $n=50$. 
		Overall, the power of all tests increases with the smoothness of \(X\) (as \(\nu \) increases), 
		but our bootstrap tests consistently outperform or remain competitive across all scenarios.
		
		The average computation times for each test across sample sizes $n \in \{50,200\}$ over 100 replications are reported in \autoref{tb_computing}.
		The proposed bootstrap tests require reasonable times for computing; the procedure takes only a few minutes even without parallelization. This is faster than the unparallelized version of the approach by \cite{LL25} and can even yield a lower computation time than the resampling-free approach of \cite{SC15}.

		\subsection{Confidence interval for a projection} \label{ssecProj}

		Inference for projections of $\beta$ is another important problem in SoFR models. 
		The projection $\langle \beta, v \rangle$ onto a direction $v \in \HH$ can be interpreted as a functional version of the classical contrast;
		it represents the contribution of $v$ to $\beta$.
		%	and it is closely related to the conditional mean response $\mu(v) \equiv \eo[Y|X=v] = \ap + \langle \beta, v \rangle$ at $v$. 
		This problem has been studied in different ways. 
		\cite{CMS07} firstly provide a CLT for $\langle \beta, v \rangle$ when $v$ is either random or non-random. 
		Both \cite{GM11} and \cite{KH16a} follows the framework of \cite{CMS07} with non-random $v$,
		in which case the assumption that 
		\begin{align}
			\sup_{j \in \N}\g_j^{-1} \langle v,\phi_j \rangle^2<\infty \label{condProjOther}
		\end{align}
		is commonly imposed.
		In simpler words, this means that the function $v$ is relatively smooth compared to the random function $X$,
		which facilitates deriving asymptotic normality results there. 
		Recently, \cite{YDN23RB} suggest a residual bootstrap for this problem with random $v$. 
		See \autoref{ssecFurtherDiscussionProjInfer} for further discussion of the existing literature on projection inference.

		In this section, we demonstrate that our main theoretical development naturally induces inference results for projection $\langle \beta, v \rangle$ outside the frameworks in these previous works.
		Namely, our result below is applicable to the non-random projection direction $v \in \HH$ such that the condition in \eqref{condProjOther} is not satisfied,
		which may not be straightforward from the classical analysis.
		To state the theorem, let $v \in \HH \bsh \{0\}$ and we introduce projection statistics
		\begin{alignat}{3} \label{eqStatProj}
			T_{\mathrm{proj},J}(v) 
			& \equiv \sqrt{n \over \hat{\tau}_J(v)} (\langle \hat{\beta}_J,  v \rangle - \langle \beta, v \rangle), 
			&& \qquad T_{\mathrm{proj},J}^*(v)
			&& \equiv \sqrt{n \over \hat{\tau}_J(v)} (\langle \hat{\beta}_J^*, v \rangle - \langle \hat{\beta}_J, v \rangle),
		\end{alignat}
		%	\begin{align*}
			%		T_{\mathrm{proj},J}(v) & \equiv \sqrt{n \over \hat{\tau}_J(v)} \langle \hat{\beta}_J - \beta, v \rangle, \\
			%		T_{\mathrm{proj},J}^*(v) & \equiv \sqrt{n \over \hat{\tau}_J(v)} \langle \hat{\beta}_J^* - \hat{\beta}_J, v \rangle,
			%	\end{align*}
		where $\hat{\tau}_J(v) \equiv \sum_{j=1}^J \hat{\g}_j^{-1} \langle v, \hat{\phi}_j \rangle^2$.

		\begin{thm} \label{thmProj}
			In addition to the assumptions for \autoref{thmMain}(a), we suppose the following:
			\begin{enumerate}[label=(P\arabic*)]
				\item $J = O ( \tau_J(v) )$, where $\tau_J(v) \equiv  \sum_{j=1}^J \g_j^{-1} \langle v, \phi_j \rangle^2$; \label{condProjScale}
				
				\item $J^{7+\dt} \asymp n$ for some $\dt \in (0,\infty)$; and \label{condProjRate}
				
				\item $J^{3+\dt/2} \max \left\{ \sum_{j>J} \langle \beta, \phi_j \rangle^2, \sum_{j>J} \langle v, \phi_j \rangle^2 \right\} = o(1)$. \label{condProjSmooth}
				
			\end{enumerate}
			Then, 
			the Gaussian and bootstrap approximations are valid for the projection statistic $T_{\mathrm{proj},J}(v)$ \eqref{eqStatProj} onto $v$ in the sense that 
			\begin{align*}
				\sup_{s \in \R} | \pr ( T_{\mathrm{proj},J}(v) \leq s |\xx_n ) - \Phi(s/\s) | & \xrightarrow{\pr} 0, \\
				\sup_{s \in \R} | \pr ( T_{\mathrm{proj},J}(v) \leq s |\xx_n ) 
				- \pr^* ( T_{\mathrm{proj},J}^*(v) \leq s |\xx_n ) | & \xrightarrow{\pr} 0.
			\end{align*}

		\end{thm}
		
		The proof of \autoref{thmProj} is provided in Appendix~\ref{appTechProj}. 
		This result provide cases where the approximations in \autoref{thmProj} hold even when $\sum_{j=1}^J \g_j^{-1} \langle v, \phi_j \rangle^2 \gtrsim J \to \infty$.
		Here, $r_j \gtrsim s_j$ means that $s_j/r_j$ is bounded above.
		The direction $v$ can be less smoother relative to $X$,
		which is not considered the current literature such as \citet[Theorem~3]{CMS07} or
		\citet[Theorem~1]{KH16a}.
		For instance, under polynomial decay with  $\g_j \asymp j^{-a}$, the direction $v$ allows to have coefficients with slower decay as $|\langle v, \phi_j \rangle| \gtrsim \g_j^{1/2} \asymp j^{-a/2}$,
		implying $\sum_{j=1}^J \g_j^{-1} \langle v, \phi_j \rangle^2 \gtrsim J \to\infty$.
		Condition~\ref{condProjScale} also can be beyond the assumption in \eqref{condProjOther}, for example, if $\g_j \asymp j^{-5}$ and $|\langle v, \phi_j \rangle| \asymp j^{-2}$, then $\sum_{j=1}^J \g_j^{-1} \langle v, \phi_j \rangle^2 \asymp J^2 \to\infty$ but $\sup_{j \in \N} \g_j^{-1} \langle v, \phi_j \rangle^2 = \infty$.
		Some other interpretations of these additional assumptions are listed. 
		Condition~\ref{condProjScale} is a deterministic version of the one $h_n t_{h_n}(X_0)^{-1} = O_\pr(1)$ assumed in \citet[Theorem~1]{YDN23RB}.
		Condition~\ref{condProjRate} simplifies the analysis, which can be relaxed, but where similar conditions are often assumed in the literature \citep{CMS07, HH07}.
		The last condition \ref{condProjSmooth} is needed to remove the bias term related to \(\sum_{j>J} \langle \beta, \phi_j \rangle \langle v, \phi_j \rangle\) 
		that appears where the projection direction $v$ interacts with the slope parameter $\beta$.

		\autoref{thmProj} provides an important inferential tool for projection $\langle\beta, v\rangle$.
		Specifically, denoting the $(1-\ap)$-quantile of $|T_{\mathrm{proj},J}^*(v)|$ by $\hat{q}_{1-\ap}(v)$ , 
		a (symmetrized) bootstrap confidence interval for $\langle \beta, v \rangle$ is constructed by
		\begin{align}
			CI(v)
			\equiv \left[ 
			\langle \hat{\beta}_J,v \rangle - \hat{q}_{1-\ap} \sqrt{\hat{\tau}_J(v) \over n}, 
			\langle \hat{\beta}_J,v \rangle + \hat{q}_{1-\ap} \sqrt{\hat{\tau}_J(v) \over n} 
			\right]. \label{eqProjCI}
		\end{align}

		We illustrate this result through numerical studies,
		where all functions are generated using the trigonometric functions in \eqref{eqBasisTri}. 
		In the settings of the studies in the previous section, 
		we generate the functional regressors $\{X_i\}_{i=1}^n$ through the Karhunen--Loève expansion
		to explicitly manipulate the decay rates:
		\begin{align} \label{eqKLexp}
			X \ed \sum_{j=1}^{100} \langle X, \phi_{\mathrm{tri},j} \rangle \phi_{\mathrm{tri},j}.
		\end{align}
		Specifically, with independent projections $\langle X, \phi_{\mathrm{tri},j} \rangle \sim \nd(0,\g_j)$, 
		we consider two eigenvalue decay rates as
		$\g_j = 100(2+j)^{-4}$ and $\g_j = 300(2+j)^{-5}$,
		and generate $\{X_i\}_{i=1}^n$ as iid copies of $X$ in \eqref{eqKLexp}. 
		The projection directions $v$ under considerations are (i) $v \equiv \sum_{j=1}^{100} j^{-2} \phi_{\mathrm{tri},j}$ and (ii) $v \in L^2([0,1])$ defined by an indicator as $v(u) = \I(0.4<u<0.6)$ for $u \in [0,1]$,
		which are noted by ``Smooth'' and ``Indicator'', respectively.
		The former direction is smoother, 
		while both do not fall into the assumption in \eqref{condProjOther} (cf.~\autoref{lemProjIndicator}). 
		Using the latter indicator direction leads naturally to a statistical test for the sign of $\int_{0.4}^{0.6} \beta(u) du$,
		which is relevant to the question of whether the slope function $\beta$ is positive or negative somewhere within the window $[0.4,0.6]$.
		With the other factors equal to the ones given \autoref{ssecTest},
		we approximate the coverage probabilities and the widths of the bootstrap interval in \eqref{eqProjCI} through Monte Carlo iterations.
		
		The empirical coverage probabilities of the 95\% bootstrap intervals for the projections \(\langle \beta, v \rangle\) are summarized in \autoref{tb_proj}, 
		along with their corresponding average widths in parentheses. 
		These results are based on a truncation level \(J\) chosen via the FVE criterion in \eqref{eqFVE} with %either \(\rho=0.75\) or
		\(\rho=0.95\). 
		In general, the coverage rates remain close to the nominal level 0.95, even for a small sample size of \(n=50\) in some scenarios. 
		%	As expected, slight undercoverage occurs in challenging scenarios,
		%	specifically when (i) the decay rate of \(\beta \) is slow, 
		%	or (ii) the decay of \(\gamma _{j}\) does not match that of \(|\langle \beta, \phi_j \rangle \langle v, \phi_j \rangle|\).
		%	Nevertheless, selecting larger truncation levels substantially improves these cases, leading to better coverage results.

		\begin{table}[htbp]
			\centering
			\caption{Empirical coverage rates of the 95\% bootstrap intervals for the projection $\langle \beta, v \rangle$
			}
			\label{tb_proj}
			\renewcommand{\arraystretch}{1.2}\medskip
			\resizebox{0.95\linewidth}{!}{\begin{tabular}{cc|ccc|ccc}
					\hline\hline
					\multicolumn{2}{c|}{$\g_j$} & \multicolumn{3}{c|}{$ 100(2+j)^{-4}$} & \multicolumn{3}{c}{$300(2+j)^{-5}$} \\
					\hline
					\multicolumn{2}{c|}{$|\langle \beta, \phi_j \rangle|$}  & $20(2+j)^{-3}$ & $200(2+j)^{-5}$ & $3e^{-j}$ & $20(2+j)^{-3}$ & $200(2+j)^{-5}$ & $3e^{-j}$ \\
					\hline
					%				\multirow{2}{*}{$\rho = 0.75$} 
					%				& $n = 50$  & 0.927 (0.760) & 0.942 (0.760) & 0.945 (0.760) & 0.916 (0.750) & 0.924 (0.758) & 0.904 (0.753) \\
					%				& $n = 200$ & 0.941 (0.385) & 0.944 (0.384) & 0.932 (0.386) & 0.913 (0.390) & 0.937 (0.388) & 0.874 (0.387) \\
					%				\hline
					\multirow{2}{*}{Smooth} 
					& $n = 50$  & 0.942 (0.816) & 0.931 (0.821) & 0.944 (0.822) & 0.931 (0.839) & 0.937 (0.845) & 0.933 (0.839) \\
					& $n = 200$ & 0.952 (0.416) & 0.942 (0.416) & 0.944 (0.417) & 0.951 (0.423) & 0.941 (0.422) & 0.934 (0.421) \\ \hline
					\multirow{2}{*}{Indicator} & $n=50$  & 0.931 (0.778) & 0.930 (0.789) & 0.924 (0.773) & 0.920 (0.649) & 0.930 (0.650) & 0.941 (0.646) \\
					& $n=200$ & 0.951 (0.469) & 0.955 (0.473) & 0.950 (0.472) & 0.951 (0.332) & 0.963 (0.332) & 0.947 (0.332) \\
					\hline\hline
			\end{tabular}}
		\end{table}
		
		We report results based on relatively large truncation levels, which are necessary because of the intricate interplay among $\beta$, $\g_j$, and $v$. 
		Although a small truncation level may be adequate for capturing the main variation in $X$, 
		it may be generally insufficient for inference on $\langle \beta, v \rangle$, since $\beta$ or $v$ may retain substantial values in later eigenfunction coefficients. 
		This illustrates that the choice of truncation level could need to be driven by the specific statistical objective, 
		rather than by variance explanation alone. 
		See the last paragraph of \autoref{sec_conc} for further discussion of truncation-level selection.

	\section{Conclusion} \label{sec_conc}
	
	It is worth emphasizing three features of our contribution: 
	(i) our main result is concise and conceptually transparent (\autoref{sec_main}),
	(ii) the testing procedures it motivates are remarkably simple to implement yet provide highly effective tools for statistical inference  (\autoref{ssecTest}),
	and 
	(iii) it also yield a new result for projection inference that goes beyond the classical framework (\autoref{ssecProj}).
	A key takeaway is that the FPCR estimator, although often regarded as a classical or even outdated approach, remains valuable 
	and, in fact, leads to powerful inferential tools for this challenging functional regression setting with infinite dimensionality. 
	We believe that the insights developed here will open the door to new avenues of research on inference for more complex functional regression models.

		We list several potential extensions. First, the ideas developed in this work can be extended to other functional regression models, including function-on-function regression \citep{CM13} and functional generalized linear models \citep{DPZ12}, 
		where inference tools for the entire slope function remain limited. 
		An extension to dependent data is also of clear interest, especially because quadratic-form test statistics are commonly studied under temporal \citep{HKR13} or spatial \citep{KKD24} dependence.
		More broadly, our findings suggest new directions for statistical inference in complex regression models with infinite-dimensional covariates, beyond standard functional regression. 
		Recent developments such as intrinsic Riemannian \citep{LY19}, Wasserstein \citep{CLM23}, and Hilbert--Schmidt \citep{CP25} regression models have gained increasing attention; 
		these approaches rely on FPCR-type estimators adapted to their respective settings and can be viewed as extensions of the FPCR estimator studied here. However, formal inferential foundations for these models are still largely unavailable. 
		The theoretical and methodological results of this paper provide a natural first step toward developing inference for regression models with density-valued, manifold-valued, or Riemannian functional predictors.
		%	including data on spheres, tori, and spaces of symmetric positive-definite matrices. 
		In a related direction, principal components analysis has recently been developed for other non-Euclidean or infinite-dimensional data objects, 
		such as infinite-dimensional spherical data \citep{Dai22}, flows of covariance operators \citep{SP26}, and replicated point processes \citep{PRRP26}. 
		These developments may lead to corresponding regression frameworks and FPCR-type estimators, for which the ideas in this paper could be adapted. 
		Although such extensions are beyond the scope of the present study, they warrant careful investigation in future work.

		As a closing remark, we discuss the choice of the truncation level $J$.
		This is still an open problem in the literature,
		and, to our knowledge, beyond the selection based on FVE, there is no formal work that studies this choice in general,
		although \cite{lei14} and \cite{HMV13} propose  adaptive choices focusing only on testing for $H_0:\beta=0$.
		The optimal choice of $J$ is likely to depend on the target task, such as hypothesis testing for $\beta$, inference for a projection $\langle \beta, v \rangle$, or prediction of the new response $Y_0$ at a new response $X_0$.
		For the testing problem, in unreported preliminary simulations, we observed that cross-validation led to unreliable inference for the problems considered here;
		for this reason, we do not include those results.
		The threshold value $\rho$ in the FVE criterion \eqref{eqFVE} does not appear to substantially affect the power performance of our tests, unless it is chosen to be extreme; 
		see \autoref{fig_test1sofr_FVE} for additional simulation results.
		For the latter two problems, a best choice for inference may depend on the direction $v$ or on the new regressor $X_0$, as discussed \cite{YDN25WB}. 
		In subsequent work, we plan to investigate this important problem in greater generality, together with associated inference procedures.

%\clearpage

%%%%%%%%%%%%%%%%%%%%%%%%%%%%%%%%%%%%%%%%%%%%%%
%% Single Appendix:                         %%
%%%%%%%%%%%%%%%%%%%%%%%%%%%%%%%%%%%%%%%%%%%%%%
%\begin{appendix}
%\section*{???}%% if no title is needed, leave empty \section*{}.
%\end{appendix}
%%%%%%%%%%%%%%%%%%%%%%%%%%%%%%%%%%%%%%%%%%%%%%
%% Multiple Appendixes:                     %%
%%%%%%%%%%%%%%%%%%%%%%%%%%%%%%%%%%%%%%%%%%%%%%
\begin{appendix}

\section{Preliminaries} \label{appTechPre}

This appendix mainly collects the technical details supporting the theorems.
It is organized into six parts. 
After introducing some notation in \autoref{appTechPre},
\autoref{appTechMain} establishes our main theorems including the Gaussian and bootstrap approximations. 
%The subsequent sections, \autoref{appTechVar} and \autoref{appTechBias}, contain the proofs for the variance and bias components, respectively.
The subsequent \autoref{appTechBias} includes functional calculus theory, 
while \autoref{appTechLem} gathers several ancillary lemmas used throughout the analysis.
Finally, in \autoref{secAddDiscuss}, we further discuss several theoretical points,
and  we provide extra simulation results in \autoref{appExtraSim} to conclude the Appendix.

We introduce new notation for further developments.
The cross-covariance functions of the regressors $\{X_i\}_{i=1}^n$ with the responses $\{Y_i\}_{i=1}^n$ and the errors $\{\e_i\}_{i=1}^n$ are defined by
\begin{alignat*}{3}
	\widehat{\Dt} 
	& \equiv n^{-1} \sum_{i=1}^n (X_i - \bar{X})(Y_i - \bar{Y}), 
	&& \qquad U_n 
	&& \equiv n^{-1} \sum_{i=1}^n (X_i - \bar{X})(\e_i - \bar{\e}),
\end{alignat*}
where $\bar{\e} \equiv n^{-1} \sum_{i=1}^n \e_i$. 
Their bootstrap counterparts are
\begin{alignat*}{3}
	\widehat{\Dt}^*
	& \equiv n^{-1} \sum_{i=1}^n (X_i - \bar{X})(Y_i^* - \bar{Y}^*), 
	&& \qquad U_n^* 
	&& \equiv n^{-1} \sum_{i=1}^n (X_i - \bar{X})(\e_i^* - \bar{\e}^*),
\end{alignat*}
where $\bar{Y}^* \equiv n^{-1} \sum_{i=1}^n Y_i^*$ and $\bar{\e}^* \equiv n^{-1} \sum_{i=1}^n \e_i^*$. 
At both levels, these quantities are related as 
\begin{align*}
	\widehat{\Dt} = \widehat{\ga}\beta + U_n
	\quad
	\text{and}
	\quad
	\widehat{\Dt}^* = \widehat{\ga}\hat{\beta}_J + U_n^*,
\end{align*}
respectively.
Also, for each $a \in \R$, we define the pseudo $a$-th powers of $\ga$ and $\widehat{\ga}$ by
\begin{align*}
	\ga_J^a 
	\equiv \sum_{j=1}^J \g_j^a \phi_j^{\otimes 2},
	\quad
	\widehat{\ga}_J^a 
	\equiv \sum_{j=1}^J \hat{\g}_j^a \hat{\phi}_j^{\otimes 2},
\end{align*}
respectively.
When $a=0$, these represent Riesz projection operators respectively given by 
\begin{align*}
	\Pi_J \equiv \ga_J^0 = \sum_{j=1}^J \phi_j^{\otimes 2},
	\quad 
	\widehat{\Pi}_J
	\equiv \widehat{\ga}_J^0
	\equiv \sum_{j=1}^J \hat{\phi}_j^{\otimes 2}.
\end{align*}

We start with a simple observation in \autoref{propFPCRequiv}, 
as we have not found an exact reference.
This explicitly guarantees that the slope estimators in various papers based on FPCs (e.g., \cite{CFS99, CMS07} versus \cite{KSM16} and \citet[Equation~(8.19)]{HK12}) are in fact the same.
Recall that the projected regressors are denoted by $\bm{X}_i \equiv [ \langle X_i, \hat{\phi}_j \rangle ]_{1 \leq j \leq J}$ for $i=1,\dots,n$.

\begin{prop} \label{propFPCRequiv}
	The fitted coefficients from the linear regression of $\{Y_i\}_{i=1}^n$ on $\{\bm{X}_i\}_{i=1}^n$
	can be represented as
	$\hat{\beta}_{\mathrm{coef},j} \equiv \hat{\g}_j^{-1} \langle \widehat{\Dt}, \hat{\phi}_j \rangle$ for $j =1,\dots,J$,
	further providing the following expression for the FPCR estimator $\hat{\beta}_J$ in \eqref{eq_fpcr_est}:
	\begin{align*}
		\hat{\beta}_J 
		\equiv \widehat{\ga}_J^{-1} \widehat{\Dt}
		= \sum_{j=1}^J \hat{\g}_j^{-1} \langle \widehat{\Dt}, \hat{\phi}_j \rangle \hat{\phi}_j.
	\end{align*}
\end{prop}

\begin{proof}
	Let
	\begin{alignat*}{3}
		\widehat{\bm{\ga}}_J 
		& \equiv n^{-1} \sum_{i=1}^n (\bm{X}_i - \bar{\bm{X}}) (\bm{X}_i - \bar{\bm{X}})^\top, 
		&& \qquad \widehat{\bm{\Dt}}
		&& \equiv n^{-1} \sum_{i=1}^n (\bm{X}_i - \bar{\bm{X}}) (Y_i - \bar{Y})
	\end{alignat*}
	denote the $J\times J$ covariance matrix of $\{\bm{X}_i\}_{i=1}^n$ and the $J$-dimensional cross-covariance vector between $\{\bm{X}_i\}_{i=1}^n$ and $\{Y_i\}_{i=1}^n$, respectively,
	where $\bar{\bm{X}} \equiv n^{-1} \sum_{i=1}^n \bm{X}_i = [ \langle \bar{X}, \hat{\phi}_j \rangle ]_{1 \leq j \leq J}$.
	Then, the least squares estimator from the regression of $\{Y_i\}_{i=1}^n$ on $\{\bm{X}_i\}_{i=1}^n$ is then
	\begin{align*}
		\hat{\bm{\beta}}_{\mathrm{coef},J}
		= \widehat{\bm{\ga}}_J^{-1} \widehat{\bm{\Dt}}.
	\end{align*}
	Note that the covariance matrix $\widehat{\bm{\ga}}_J $ is a diagonal matrix with the first $J$ eigenvalues being diagonal elements,
	as
	\begin{align*}
		\widehat{\bm{\ga}}_J  
		%		& = (\bm{X}_i-\bar{\bm{X}}) (\bm{X}_i-\bar{\bm{X}})^\top
		& = n^{-1} \sum_{i=1}^n \left( [ \langle X_i - \bar{X}, \hat{\phi}_j \rangle ]_{1 \leq j \leq J} \right) \left( [ \langle X_i - \bar{X}, \hat{\phi}_j \rangle ]_{1 \leq j \leq J} \right)^\top
		%		\\& = n^{-1} \sum_{i=1}^n [\langle X_i - \bar{X}, \hat{\phi}_j \rangle \langle X_i - \bar{X}, \hat{\phi}_{j'} \rangle]_{1 \leq j, j' \leq J}
		\\& =  n^{-1} \sum_{i=1}^n [ \langle (\langle X_i - \bar{X})^{\otimes 2} \hat{\phi}_j, \hat{\phi}_{j'} \rangle]_{1 \leq j, j' \leq J}
		= [\langle \widehat{\ga}  \hat{\phi}_j, \hat{\phi}_{j'} \rangle]_{1 \leq j, j' \leq J}
		= \mathrm{diag}([\hat{\g}_j]_{1 \leq j \leq J}).
	\end{align*}
	Similarly, the cross-covariance vector $\widehat{\bm{\Dt}}$ is expressed as
	\begin{align*}
		\widehat{\bm{\Dt}}
		& = n^{-1} \sum_{i=1}^n \left( [ \langle X_i - \bar{X}, \hat{\phi}_j \rangle ]_{1 \leq j \leq J} \right) (Y_i-\bar{Y})
		%		\\& = n^{-1} \sum_{i=1}^n \left( [ \langle (X_i - \bar{X})(Y_i-\bar{Y}), \hat{\phi}_j \rangle ]_{1 \leq j \leq J} \right) 
		= [ \langle \widehat{\Dt}, \hat{\phi}_j \rangle ]_{1 \leq j \leq J}. 
	\end{align*}
	Combining these two, we have $\hat{\bm{\beta}}_{\mathrm{coef},J} = [ \hat{\g}_j^{-1} \langle \widehat{\Dt}, \hat{\phi}_j \rangle ]_{1 \leq j \leq J}$, deriving the desired result.
\end{proof}

We introduce operator norms, $\opnorm{\cdot}_\infty$, $\opnorm{\cdot}_2$, and $\opnorm{\cdot}_1$,
which will be used throughout the Appendix.
For a bounded linear operator $\ttt$ on $\HH$, 
the supremum norm $\opnorm{\cdot}_\infty$ defined by
\begin{align*}
	\opnorm{\ttt}_\infty & \equiv \sup_{x \in \HH:\|x\|=1} \|\ttt x\|.
\end{align*}
For a compact operator $\ttt$ on $\HH$ with eigenvalues $\{\ld_j\}_{j=1}^\infty$ \cite[cf.][Chapter~4]{HE15}, 
its Hilbert--Schmidt and nuclear norms are defined by
\begin{alignat*}{3}
	\opnorm{\ttt}_2^2 
	& \equiv \sum_{j=1}^\infty \ld_j^2, 
	&& \qquad \opnorm{\ttt}_1
	&& \equiv \sum_{j=1}^\infty |\ld_j|,
\end{alignat*}
respectively,
when the sums are finite.
These norms are related as $\opnorm{\cdot}_\infty \leq \opnorm{\cdot}_2 \leq \opnorm{\cdot}_1$;
see \citet[Equation~(1.55)]{Bosq00} for example.

In what follows, we assume that $\eo[X] = 0$ and $\eo[Y]=0$, implying $\ap=0$, which does not lose generality for theoretical development.

%The operator-scaled functional statistic $T_J$ in \eqref{eqStat} is decomposed as
%\begin{align}
%	\sqrt{n/J} \widehat{\ga}^{1/2} (\hat{\beta}_J - \beta)
%	& = \sqrt{n/J}  \ga_J^{-1/2} U_n  \nonumber
%	\\& \hspace{12pt} + \sqrt{n/J} (\widehat{\ga}_J^{1/2} - \ga_J^{1/2}) (\widehat{\ga}_J^{-1} - \ga_J^{-1}) U_n 
%	+ \sqrt{n/J} \ga_J^{1/2} (\widehat{\ga}_J^{-1} - \ga_J^{-1}) U_n \nonumber
%	\\& \hspace{12pt} + \sqrt{n/J}  (\widehat{\ga}_J^{1/2} - \ga_J^{1/2}) (\widehat{\Pi}_J - \Pi_J) \beta
%	+ \sqrt{n/J}  \ga_J^{1/2} (\widehat{\Pi}_J - \Pi_J) \beta \nonumber
%	\\& \hspace{12pt} + \sqrt{n/J}  (\widehat{\ga}_J^{1/2}-\ga_J^{1/2}) (\Pi_J \beta - \beta) 
%	+ \sqrt{n/J}  \ga_J^{1/2} (\Pi_J \beta - \beta). \label{eqStatDecomp}
%\end{align}
%We will derive the Gaussian approximation of the first term in \eqref{eqStatDecomp} in the conditional Wasserstein distance and the convergence in zero of the other terms.

\section{Proof of the theorems} \label{appTechMain}

\subsection{Proof of \autoref{thmMain}} \label{appTechMain1}
To investigate our approximation results, 
we see how the original and bootstrap statistics are decomposed. 
At the data-level, 
using $\widehat{\Dt} = \widehat{\ga}\beta + U_n$, 
the FPCR estimator is decomposed as
\begin{align*}
	\hat{\beta}_J - \beta 
	= \widehat{\ga}_J^{-1} \widehat{\Dt} - \beta
	= \widehat{\ga}_J^{-1} U_n + (\widehat{\Pi}_J-I) \beta,
\end{align*}
where these terms serve as the variance and bias parts in the asymptotic analysis; see \citet[Equation~(11)]{CMS07} and \citet[Equation~(11)]{YDN23RB} for similar decompositions.
Unlike these previous works, our operator-scaling $\widehat{\ga}_J^{1/2}$ not only stabilize the variance term $\widehat{\ga}_J^{-1} U_n$ but also deterministically remove the bias term as $\widehat{\ga}_J^{1/2}(\widehat{\Pi}_J-I) = 0$. 
On the other hand, by construction, the bootstrap FPCR estimator does not leave any bias term regardless of scaling,
as
\begin{align*}
	\hat{\beta}_J^* - \hat{\beta}_J
	& = \widehat{\ga}_J^{-1} \widehat{\Dt}^* - \hat{\beta}_J
	= \widehat{\ga}_J^{-1} U_n^* + \widehat{\Pi}_J\hat{\beta}_J - \hat{\beta}_J,
	\\& = \widehat{\ga}_J^{-1} U_n^*.
\end{align*}
To summarize, 
the operator-scaled functional statistics $T_J$ in \eqref{eqStat} and its bootstrap counterpart $T_J^*$ in \eqref{eqStatBTS} are represented as
\begin{align}
	T_J 
	& \equiv \sqrt{n/J} \widehat{\ga}_J^{1/2} (\hat{\beta}_J - \beta)
	= \sqrt{n/J}  \widehat{\ga}_J^{-1/2} U_n, \label{eqDecompStat} \\
	T_J^*
	& \equiv \sqrt{n/J} \widehat{\ga}_J^{1/2} (\hat{\beta}_J^* - \hat{\beta}_J)
	%	& = \sqrt{n/J} \widehat{\ga}_J^{1/2} (\widehat{\ga}_J^{-1} U_n^* + \widehat{\Pi}_J \hat{\beta}_J - \hat{\beta}_J)  
	= \sqrt{n/J}  \widehat{\ga}_J^{-1/2} U_n^*, \label{eqDecompStatBTS}
\end{align}
where both consists only of their respective variance terms.
Let $Q$ denote the common distribution of the errors $\{\e_i\}_{i=1}^n$
and $\widehat{Q}$ represent the (discrete) uniform distribution on $\{\hat{\e}_i\}_{i=1}^n$. 
With this background, the proofs of the theorems are given next.

\begin{proof}[Proof of \autoref{thmMain}(a)]	
	To derive the Gaussian approximation, 
	observing that $$\widehat{\ga}_J^{-1/2} U_n  = \sum_{j=1}^J \hat{\g}_j^{-1/2} \langle U_n, \hat{\phi}_j \rangle \hat{\phi}_j,$$
	we investigate the behavior of the coefficients $[\hat{\g}_j^{-1/2} \langle U_n, \hat{\phi}_j \rangle]_{1 \leq j \leq J}$ instead of $\ga_J^{-1/2} U_n$ itself. 
	For $i=1,\dots,n$, we write 
	\begin{align*}
		Z_{ij} = \hat{\g}_j^{-1/2} \langle X_i-\bar{X}, \hat{\phi}_j \rangle\e_i,
		\quad j=1,\dots,J,
		\quad
		\bm{Z}_{iJ} \equiv [\hat{\g}_j^{-1/2} \langle X_i-\bar{X}, \hat{\phi}_j \rangle \e_i]_{1 \leq j \leq J} 
	\end{align*}
	so that $\sqrt{n}\bar{\bm{Z}}_J = [\hat{\g}_j^{-1/2}\langle \sqrt{n}U_n, \hat{\phi}_j \rangle]_{1 \leq j \leq J}$, where $\bar{\bm{Z}}_J \equiv n^{-1} \sum_{i=1}^n \bm{Z}_{iJ}$. 
	Then, we have $\eo^X[\bm{Z}_{iJ}] = \bm{0}$ and $n^{-1} \sum_{i=1}^n \var^X[\bm{Z}_{iJ}] = \s^2 \bm{I}_J$.
	Also, we derive
	\begin{align*}
		\eo^X[\|\bm{Z}_{iJ}\|_{\R^J}^4]
		& = \eo^X \left[ \left( \e_i^2 \sum_{j=1}^J \hat{\g}_j^{-1} \langle X_i - \bar{X},\hat{\phi}_j \rangle^2 \right)^2 \right]
		\\& = \eo^X[\e_1^4] \langle \widehat{\ga}_J^{-1} (X_i - \bar{X}), (X_i - \bar{X}) \rangle^2
		\\& \leq 2\eo^X[\e_1^4] \langle (\widehat{\ga}_J^{-1}-\ga_J^{-1})(X_i - \bar{X}), (X_i - \bar{X}) \rangle^2
		\\& \hspace{12pt} + 2\eo^X[\e_1^4] \langle \ga_J^{-1}(X_i - \bar{X}), (X_i - \bar{X}) \rangle^2
		\\& \leq 2\eo^X[\e_1^4] \left\{ \opnorm{\widehat{\ga}_J^{-1}-\ga_J^{-1}}_\infty^2 \|X_i - \bar{X}\|^4 + \left( \sum_{j=1}^J \g_j^{-1} \langle X_i - \bar{X}, \phi_j \rangle^2 \right)^2 \right\}
		\\& \leq 2\eo^X[\e_1^4] \left\{ 8\opnorm{\widehat{\ga}_J^{-1}-\ga_J^{-1}}_\infty^2 (\|X_i\|^4 + \|\bar{X}\|^4) + J \sum_{j=1}^J \g_j^{-2} \langle X_i - \bar{X}, \phi_j \rangle^4 \right\}.
	\end{align*}
	
	Using the assumptions, 
	recall from \autoref{lemPerturb2} that 
	$\opnorm{\widehat{\ga}_J^{-1} - \ga_J^{-1}}_\infty = D_{1n} + D_{2n}$, where $\eo[D_{1n}] = o(n^{-1/2} J^3 (\log J)^2)^{1/2}$ and $r_nD_{2n} = o_\pr(1)$ for any random quantity $r_n$. 
	Note that the term $D_{2n}$ absorbs any terms, thus not affecting our results \cite[e.g.,][]{CMS07}.
	Therefore, for simplicity, we omit this term from the subsequent analysis and treat $\opnorm{\widehat{\ga}_J^{-1} - \ga_J^{-1}}_\infty$ as $o_\pr(n^{-1/2} J^3 (\log J)^2)^{1/2}$.
	Consequently, 
	we have that
	\begin{align*}
		n^{-1} \sum_{i=1}^n \eo^X[\|\bm{Z}_{iJ}\|_{\R^J}^4] 
		& = O_\pr \left( \opnorm{\widehat{\ga}_J^{-1}-\ga_J^{-1}}_\infty^2 + J^2 \right)
		\\& = O_\pr(n^{-1/2} J (\log J)^2 + 1)J^2)
		= O_\pr(J^2)
		, \\
		n^{-1} \sum_{i=1}^n \eo^X[\|\bm{Z}_{iJ}\|_{\R^J}^4]^2
		& = O_\pr \left( \opnorm{\widehat{\ga}_J^{-1}-\ga_J^{-1}}_\infty^2 + J^2 \right)^2
		\\& = \{O_\pr(n^{-1/2} J (\log J)^2 + 1)J^2\}^2
		= O_\pr(J^4),
	\end{align*}
	since $n^{-1/2} J (\log J)^2 = o(1)$ due to Condition~\ref{condBasicTrunc}.
	%	if $\opnorm{\widehat{\ga}_J^{-1}-\ga_J^{-1}}_\infty = o_\pr(1)$. 
	Then, using \citet[Theorem~5]{Bonis20}, 
	we find that
	\begin{align*}
		& \wsd^X(\sqrt{n}\bar{\bm{Z}}_J, \s \bm{G}_J)
		\\\leq& 
		%		{n^{1/2} \{O_\pr(J^2) + O_\pr(D_{2n})\}^{1/2} \over n}
		%		\\& + \left[ \{ O_\pr(J^2) + O_\pr(D_{2n}) \}^{1/2}
		%		\{O_\pr(J^2) + O_\pr(D_{2n}) + O_\pr(J^4) + O_\pr(D_{2n})\}^{1/2} \over n \right]^{1/2}
		{O_\pr(nJ^2)^{1/2} \over n} + \left[  {O_\pr(J^2)^{1/2} \{O_\pr(J^2) + O_\pr(J^4)\}^{1/2} \over n} \right]^{1/2}
		\\=& O_\pr(n^{-1/2}J) + O_\pr(n^{-1/2}J^{3/2})
		= O_\pr(n^{-1/2}J^{3/2}),
	\end{align*}
	where $\bm{G}_J \sim \nd(\bm{0},\bm{I}_J)$ is a $J$-dimensional Gaussian random vector. 
	This yields the desired result as
	\begin{align}
		\wsd^X( \sqrt{n/J}\widehat{\ga}_J^{-1/2} U_n, \s J^{-1/2} G_J )
		= J^{-1/2} \wsd^X(\sqrt{n}\bar{\bm{Z}}_J, \s \bm{G}_J)
		= O_\pr(n^{-1/2} J). \label{eq_thmGau_rate}
	\end{align}
	%	It follows because
	%	\begin{align*}
		%		\wsd^X(T_J, \s J^{-1/2}G_J)^2
		%		&\leq 2 \wsd^X( \sqrt{n/J}  \widehat{\ga}_J^{-1/2} U_n, \s J^{-1/2}G_J)^2 + 2\eo^X[\|B_n\|^2]
		%		\\& \leq 2 O_\pr(n^{-1}J^2) + 2\|B_n\|^2 = o_\pr(1), \\
		%		\wsd^X(T_J, T_J^*)^2
		%		&\leq 2 \wsd^X( \sqrt{n/J}  \widehat{\ga}_J^{-1/2} U_n, \sqrt{n/J}  \widehat{\ga}_J^{-1/2} U_n^*)^2 + 2\eo^X[\|B_n\|^2]
		%		\\& \leq 2 \wsd^X(Q,\widehat{Q})^2 + 2\|B_n\|^2 = o_\pr(1),
		%	\end{align*}
	%	by Propositions~\ref{propGauVar}--\ref{propBias}.
\end{proof}

\begin{proof}[Proof of \autoref{thmMain}(b)]
	The bootstrap approximation follows from the following two results, along with the representations in \eqref{eqDecompStat}--\eqref{eqDecompStatBTS}:
	\begin{align}		
		\wsd^X( \sqrt{n/J}  \widehat{\ga}_J^{-1/2} U_n, \sqrt{n/J}  \widehat{\ga}_J^{-1/2} U_n^*)
		& \leq \wsd^X(Q,\widehat{Q}), \label{eq_thmMainBTS1} \\		
		\wsd^X(Q,\widehat{Q})^2 
		& = o_\pr(1). \label{eq_thmMainBTS2}
	\end{align}
	Here, the inequality \eqref{eq_thmMainBTS1} holds deterministically,
	while the result in \eqref{eq_thmMainBTS2} requires the assumptions.
	In particular, the the inequality \eqref{eq_thmMainBTS1} is one of our key results. 
	It shows that the Wasserstein distance between the variance terms in \eqref{eqStat}–\eqref{eqStatBTS} is controlled by the Wasserstein distance between the bootstrap and true error distributions, $\widehat{Q}$ and $Q$. 
	This reduction is crucial for our analysis, 
	as it does not require stringent assumptions.
	The convergence in \eqref{eq_thmMainBTS2} says the consistency of the bootstrap error distribution $\widehat{Q}$ for the original error distribution $Q$.
	Even though a related result appears in \cite{YDN23RB},
	our version is substantially refined and provides a clearer convergence rate, which justifies presenting it here.

	To derive the inequality in \eqref{eq_thmMainBTS1},
	we follow the argument in the proof of \citet[Theorem~4.1]{KF92}.
	Since the infimum in the Wasserstein distance is attained \cite[Lemma~8.1]{Mallows72},
	there exist independently and identically distributed pairs of random variables $\{(\e_i', {\e_i^*}')\}_{i=1}^n$ such that
	\begin{enumerate}[(i)]
		\item $\e_i' \sim Q$, ${\e_i^*}' \sim \widehat{Q}$,
		
		\item $(\e_i', {\e_i^*}')$ is independent of $X_i$, and
		
		\item $\eo^X[(\e_i' - {\e_i^*}')^2] = \eo[(\e_i' - {\e_i^*}')^2] = \wsd(Q, \widehat{Q})^2$.
		
	\end{enumerate}
	We write 
	\begin{alignat*}{3}
		U_n'
		& \equiv n^{-1} \sum_{i=1}^n (X_i - \bar{X})(\e_i' - \bar{\e}'),  
		&& \qquad {U_n^*}' 
		&& \equiv n^{-1} \sum_{i=1}^n (X_i - \bar{X}) ({\e_i^*}' - {\bar{\e}^*}{'}),
	\end{alignat*}
	where $\bar{\e}' \equiv n^{-1} \sum_{i=1}^n \e_i'$ 
	and ${\bar{\e}^*}{'} \equiv n^{-1} \sum_{i=1}^n {\e_i^*}'$. 
	Using \autoref{lemAlge} and the properties in (i)--(iii),
	we have
	\begin{align*}
		n\wsd^X(\widehat{\ga}_J^{-1/2} U_n, \widehat{\ga}_J^{-1/2} U_n^*)^2
		& \leq n\eo^X[\|\widehat{\ga}_J^{-1/2}(U_n'-U_n^{*'})\|^2]
		\\& = n^{-1}\sum_{i=1}^n \|\widehat{\ga}_J^{-1/2} (X_i-\bar{X})\|^2 \eo^X[(\e_i' - {\e_i^*}')^2]
		\\& = n^{-1}\sum_{i=1}^n \|\widehat{\ga}_J^{-1/2} (X_i-\bar{X})\|^2 \wsd^X(Q,\widehat{Q})^2
		\\& = \opnormbig{n^{-1}\sum_{i=1}^n \{(\widehat{\ga}_J^{-1/2} (X_i-\bar{X})\}^{\otimes 2}}_1 \wsd^X(Q,\widehat{Q})^2
		\\& = \opnorm{\widehat{\Pi}_J}_1 \wsd^X(Q,\widehat{Q})^2
		\\& = J \wsd^X(Q,\widehat{Q})^2,
	\end{align*}
	where $\widehat{\Pi}_J \equiv \sum_{j=1}^J \hat{\phi}_j^{\otimes 2}$.

	To derive the result in \eqref{eq_thmMainBTS2}, let $R$ denote the empirical distribution of the errors $\{\e_i\}_{i=1}^n$. 
	First, \citet[Theorem~1]{Bonis20} indicates that $\wsd^X(Q, R) = \wsd(Q, R) \leq O(n^{-1/2})$.
	Recall that the errors and the residuals are represented as
	\begin{align*}
		\e_i
		& = Y_i - \eo[Y] - \langle \beta, X_i - \eo[X] \rangle, \\
		\hat{\e}_i 
		& = Y_i - \bar{Y} - \langle \hat{\beta}_J, X_i - \bar{X} \rangle,
	\end{align*}
	yielding their difference as
	\begin{align*}
		\e_i - \hat{\e}_i
		= \bar{Y}-\eo[Y] 
		+ \langle \hat{\beta}_J - \beta, X_i - \bar{X} \rangle 
		- \langle \beta, \bar{X} - \eo[X] \rangle .
	\end{align*}
	Then, we have
	\begin{align*}
		\wsd^X(R,\widehat{Q})^2
		& \leq n^{-1} \sum_{i=1}^n (\e_i - \hat{\e}_i)^2
		%		= n^{-1} \sum_{i=1}^n (\e_i - \hat{\e}_i)^2
		\\& \leq 3 (\bar{Y}-\eo[Y] )^2
		+ 3 (\langle \beta, \bar{X}-\eo[X] \rangle)^2
		+ 3 n^{-1} \sum_{i=1}^n \langle \hat{\beta}_J - \beta, X_i - \bar{X} \rangle^2 
		\\& \leq 3 \bar{\e}^2 + 3\|\beta\|^2 \|\bar{X}-\eo[X]\|^2 
		+ 3 \langle \widehat{\ga}(\hat{\beta}_J - \beta), \hat{\beta}_J - \beta \rangle
		\\& = O_\pr(n^{-1}) +  3 \| \widehat{\ga}^{1/2}(\hat{\beta}_J - \beta)\|^2
		%		\\& = O_\pr(n^{-1}) + O_\pr(n^{-1}) + O_\pr(J/n) + O_\pr(J\|B_n\|^2/n).
	\end{align*}
	where the last inequality are derived by the classical law of large numbers and CLT for Hilbertian elements \cite[e.g.,][Theorems~7.7.2 and~7.7.7]{HE15}.
	Using $\hat{\beta}_J-\beta = \widehat{\ga}_J^{-1}U_n + (\widehat{\Pi}_J-I)\beta$ and \autoref{lemQuadBound} below, 
	we have
	\begin{align*}
		\| \widehat{\ga}^{1/2}(\hat{\beta}_J - \beta)\|^2
		& \leq 2\|\widehat{\ga}_J^{-1/2}U_n\|^2 + 2\|\widehat{\ga}^{1/2} (\widehat{\Pi}_J-I)\beta \|^2
		\\& = O_\pr(J/n) + o_\pr(n^{-1/2}) + O \left( \sum_{j>J} \g_j \langle \beta, \phi_j \rangle^2 \right),
	\end{align*}
	since
	\begin{align*}
		\|\widehat{\ga}^{1/2} (\widehat{\Pi}_J-I)\beta \|^2
		=& \langle \widehat{\ga} (\widehat{\Pi}_J-I)\beta, (\widehat{\Pi}_J-I)\beta \rangle
		\\ \leq & |\langle (\widehat{\ga}-\ga) (\widehat{\Pi}_J-I)\beta, (\widehat{\Pi}_J-I)\beta \rangle|
		+ \|\ga^{1/2} (\widehat{\Pi}_J-I)\beta\|^2
		\\ \leq & \opnorm{\widehat{\ga}-\ga}_\infty (2\|(\widehat{\Pi}_J-\Pi_J)\beta\|^2 + 2\|(\Pi_J-I)\beta\|^2)
		\\& + 2 \|\ga^{1/2} (\widehat{\Pi}_J-\Pi_J)\beta\|^2
		+ 2 \|\ga^{1/2} (\Pi_J-I)\beta\|^2
		\\=& O_\pr(n^{-1/2}) O_\pr \left( \left\{ (n^{-1/2}J^2 \log J)^2 + \sum_{j>J} \langle \beta, \phi_j \rangle^2 \right\} \right)
		\\& + O_\pr(J/n) + O \left( \sum_{j>J} \g_j \langle \beta, \phi_j \rangle^2 \right),
	\end{align*}
	where Lemmas~\ref{lemPerturb1}--\ref{lemPerturb2}(a) are applied above. 
	Overall, the bootstrap error consistency holds 
	as
	\begin{align} \label{eq_thmBTSrate}
		\wsd^X(Q,\widehat{Q})^2 = O_\pr(n^{-1/2}) + O_\pr(J/n) + O \left( \sum_{j>J} \g_j \langle \beta, \phi_j \rangle^2 \right).
	\end{align}
\end{proof}

\begin{lem} \label{lemQuadBound}
	Suppose that $\s^2 = \eo[\e^2|X] \equiv \eo[\e^2]  \in (0,\infty)$.
	Then, we have 
	\begin{align*}
		\eo^X[\|\widehat{\ga}_J^{-1/2}U_n\|^2] = \s^2 J/n.
	\end{align*}
\end{lem}

\begin{proof}
	It follows that
	\begin{align*}
		\eo^X[\|\widehat{\ga}_J^{-1/2}U_n\|^2]
		= \sum_{j=1}^J \hat{\g}_j^{-1} \eo^X[\langle U_n, \hat{\phi}_j \rangle^2]
		= \s^2 J/n,
	\end{align*}
	since
	\begin{align*}
		\eo^X[\langle U_n, \hat{\phi}_j \rangle^2]
		& = \eo \left[ \left\langle n^{-1} \sum_{i=1}^n \e_i (X_i-\bar{X}), \hat{\phi}_j \right\rangle^2 \right]
		= \eo^X \left[ \left( n^{-1} \sum_{i=1}^n\langle  X_i-\bar{X}, \hat{\phi}_j \rangle  \e_i \right)^2 \right]
		\\& = n^{-2} \sum_{i=1}^n \langle X_i-\bar{X}, \hat{\phi}_j \rangle^2 \eo[\e^2]
		= \s^2 n^{-1} \langle \widehat{\ga} \hat{\phi}_j, \hat{\phi}_j \rangle
		= \s^2 n^{-1} \hat{\g}_j.
	\end{align*}
\end{proof}

\subsection{Proof of \autoref{thmTest}} \label{appTechTest}

%\begin{proof}[Proof of \autoref{thmTest}]
On one hand, when $\HH = L^2([0,1])$, the map $\Psi_{\mathrm{norm}}:\HH\to\R$ defined by $\Psi_{\mathrm{norm}}(x) \equiv \|x\| = \left\{ \int_0^1 x(u)^2 du \right\}^{1/2}$ for $x \in L^2([0,1])$ is Lipschitz continuous with constant 1.
On the other hand, when $\HH = \W^{1,2}([0,1])$, by \autoref{lemSobolevEmb},
the map $\Psi_{\mathrm{sup}}:\HH\to\R$ defined by $\Psi_{\mathrm{sup}}(x) \equiv \sup_{u \in [0,1]} |x(u)|$ for $x \in \HH$ is Lipschitz continuous on $\W^{1,2}([0,1])$ with constant $\sqrt{2}$.
\autoref{lemWSDcont} then yield
\begin{align*}
	\wsd^X(S_{l,J},\Psi_l(\s J^{-1/2}G_J)) 
	& \leq C \cdot \wsd^X(T_J, \s J^{-1/2}G_J), \\
	\wsd^X(S_{l,J}^*,\Psi_l(\s J^{-1/2}G_J))
	& \leq C \cdot \wsd^X(T_J^*, \s J^{-1/2}G_J),
\end{align*}
for $l =\mathrm{sq}$ when $\HH = L^2([0,1])$ and for $l=\mathrm{sup}$ when $\HH = \W^{1,2}([0,1])$,
where $C$ is a generic constant.
Finally, 
%	Propositions~\ref{propGauVar}--\ref{propBias} 
\autoref{thmMain} and \autoref{lemKDbddWD} along with Lemmas~\ref{lemSqDensity}--\ref{lemSupDensity} conclude the desired result.
For instance, 
when $l=\mathrm{sup}$ and $\HH = \W^{1,2}([0,1])$,
using \autoref{lemSupDensity} (and \autoref{lemKDbddWD}), 
we bound the Kolmogorov distance by the Wasserstein distance as
\begin{align*}
	& \sup_{s \in \R} |\pr(S_{\mathrm{sup},J} \leq s |\xx_n) - \pr(\Psi_{\mathrm{sup}}(\s J^{-1/2}G_J) \leq s )|^2
	\\ \leq & \sup_{s \in \R} 
	\left| \pr \left( \sup_{u \in [0,1]} J^{1/2}|T_J(u)| \leq s \Big| \xx_n \right) 
	- \pr \left( \sup_{u \in [0,1]} |\s G_J(u)| \leq s \right) \right|^2
	\\ \leq & C J^{1/4} \wsd^X(J^{1/2} T_J, \s G_J)
	= C J^{3/4} \wsd^X(T_J, \s J^{-1/2}G_J)
	\\=& O_\pr(n^{-1/2}J^{7/4}).
\end{align*}
The bootstrap counterpart can be derived by 	
replacing $J^{3/4}\wsd^X(T_J, \s J^{-1/2}G_J)$
with 
\begin{align*}
	J^{3/4}\wsd^X(T_J, \s J^{-1/2}G_J) + J^{3/4}\wsd^X(T_J, T_J^*)
	= O_\pr(n^{-1/2}J^{7/4}) + O \left( J^{3/4} \sum_{j>J} \g_j \langle \beta, \phi_j \rangle^2 \right).
\end{align*}
Similarly, the approximation results hold for the $L^2$ norm statistics using \autoref{lemSqDensity} when  $l =\mathrm{sq}$ and $\HH = L^2([0,1])$
as 
\begin{align*}
	& \sup_{s \in \R} |\pr(S_{\mathrm{sq},J} \leq s |\xx_n) - \pr(\s^2 J^{-1}\|G_J\|^2) \leq s )|^2
	\\=& \sup_{s \in \R} |\pr(S_{\mathrm{sq},J}^{1/2} \leq s |\xx_n) - \pr(\Psi_{\mathrm{norm}}(\s J^{-1/2}G_J) \leq s )|^2
	\\ \leq & \sup_{s \in \R} 
	\left| \pr \left( J^{1/2}\|T_J\| \leq s \Big| \xx_n \right) 
	- \pr \left( \s \|G_J\| \leq s \right) \right|^2
	\\ \leq & C J^{-1/2} \wsd^X(J^{1/2} T_J, \s G_J)
	= C \wsd^X(T_J, \s J^{-1/2}G_J)
	= O_\pr(n^{-1/2}J).
\end{align*}
Finally, we have the corresponding bootstrap result
by replacing $\wsd^X(T_J, \s J^{-1/2}G_J)$ with 
\begin{align*}
	\wsd^X(T_J, \s J^{-1/2}G_J) + \wsd^X(T_J, T_J^*)
	& = O_\pr(n^{-1/2}J) + O \left( \sum_{j>J} \g_j \langle \beta, \phi_j \rangle^2 \right).
\end{align*}
%	Similarly, we can (i) derive the bound for the bootstrap counterpart for $l=\mathrm{sup}$ when $\HH = \W^{1,2}([0,1])$ and (ii) prove the bootstrap approximation result for the $L^2$ norm statistics using \autoref{lemSqDensity} when  $l =\mathrm{sq}$ and $\HH = L^2([0,1])$.
This completes the proof. 
%\end{proof}

\subsection{Proof of \autoref{thmProj}} \label{appTechProj}

%\begin{proof}[Proof of \autoref{thmProj}]
We define $y \in \mathrm{span}(\{\hat{\phi}_j\}_{j=1}^J)$ such that $\langle y, \hat{\phi}_j \rangle = \hat{\g}_j^{-1/2} \langle v, \hat{\phi}_j \rangle$, implying $\widehat{\ga}_J^{1/2}y = \widehat{\Pi}_J v$.
Then, using  $\hat{\beta}_J-\beta = \widehat{\ga}_J^{-1}U_n + (\widehat{\Pi}_J-I)\beta$,
we see that
\begin{align}
	\langle \hat{\beta}_J-\beta, v \rangle
	& = \langle \hat{\beta}_J-\beta, \widehat{\Pi}_Jv \rangle
	+ \langle \hat{\beta}_J-\beta, (I-\widehat{\Pi}_J)v \rangle  \nonumber
	\\& = \langle \widehat{\ga}_J^{-1}U_n, \widehat{\Pi}_Jv \rangle
	- \langle (\widehat{\Pi}_J-I)\beta, (\widehat{\Pi}_J-I)v \rangle  \nonumber
	\\& = \langle \widehat{\ga}_J^{-1/2}U_n, y \rangle
	- \sqrt{J/n}B_n, \label{eq_thmProjDecomp}\\
	\langle \hat{\beta}_J^*-\hat{\beta}_J, v \rangle
	& = \langle \widehat{\ga}_J^{-1/2}U_n^*, y \rangle, \nonumber
\end{align}
where
\begin{align}
	B_n \equiv \sqrt{n/J}\langle (\widehat{\Pi}_J-I)\beta, (\widehat{\Pi}_J-I)v \rangle. \label{eq_thmProjBias}
\end{align}
We will later show that the bias term $B_n$ in \eqref{eq_thmProjBias} is negligible using perturbation theory given in \autoref{appTechBias}.

We first derive the Gaussian approximation result.
Note that $\langle G_J, y \rangle \sim \nd(0,\hat{\tau}_J(v))$, because $\var[\langle G_J, y \rangle|\xx_n] = \sum_{j=1}^J \langle y, \hat{\phi}_j\rangle^2 = \hat{\tau}_J(v)$.
Then, by \autoref{lemKDbddWD}, we have
\begin{align*}
	& \sup_{s \in \R}|\pr(T_{\mathrm{proj},J}(v) \leq s|\xx_n) - \Phi(s/\s)|^2
	\\=& \sup_{s \in \R}|\pr(T_{\mathrm{proj},J}(v) \leq s|\xx_n) - \pr (\s \langle G_J,y \rangle / \sqrt{\hat{\tau}_J(v)} \leq s|\xx_n) |^2
	\\ \leq & {1 \over \sqrt{2\pi}} \wsd^X(T_{\mathrm{proj},J}(v), \s \langle G_J,y \rangle / \sqrt{\hat{\tau}_J(v)})
	\\=& \sqrt{J \over \hat{\tau}_J(v)} \wsd^X \left( \sqrt{n/J} \langle \hat{\beta}_J-\beta, v \rangle, \s J^{-1/2} \langle G_J, y \rangle \right).
\end{align*}
Using the decomposition in \eqref{eq_thmProjDecomp}, the Wasserstein distance in the preceding display is bounded as
\begin{align*}
	\wsd^X \left( \sqrt{n/J} \langle \hat{\beta}_J-\beta, v \rangle, \s J^{-1/2} \langle G_J, y \rangle \right)^2
	\leq 2 \wsd^X \left( \sqrt{n/J} \langle \widehat{\ga}_J^{-1/2}U_n, y \rangle, \s J^{-1/2} \langle G_J, y \rangle \right)
	+ 2 B_n^2.
\end{align*}
Since the map $\HH \ni x \mapsto \langle x, y \rangle$ is Lipschitz continuous with constant $\|y\| = \|\widehat{\ga}_J^{-1/2}v\| = \hat{\tau}_J(v)$,
the Wasserstein distance in the upper bound is further bounded by using \autoref{lemWSDcont} as
\begin{align*}
	\wsd^X \left( \sqrt{n/J} \langle \widehat{\ga}_J^{-1/2}U_n, y \rangle, \s J^{-1/2} \langle G_J, y \rangle \right)
	& \leq \hat{\tau}_J(v) \wsd^X(\sqrt{n/J} \widehat{\ga}_J^{-1/2}U_n, \s J^{-1/2} G_J)
	\\& = \hat{\tau}_J(v) O_\pr(n^{-1/2}J),
\end{align*}
where the last rate comes from the proof of \autoref{thmMain} (cf.~Equation~\eqref{eq_thmGau_rate}). 
To sum up, the Kolmogorov distance is bounded as
\begin{align}
	& \sup_{s \in \R}|\pr(T_{\mathrm{proj},J}(v) \leq s|\xx_n) - \Phi(s/\s)|^2 \nonumber
	\\ \leq & C \left( J^{1/2} \wsd^X(\sqrt{n/J} \widehat{\ga}_J^{-1/2}U_n, \s J^{-1/2} G_J) + {J \over \hat{\tau}_J(v)} B_n^2 \right)^{1/4} \label{eq_thmProjRate}
	\\=& C \{J^{1/2} O_\pr(n^{-1/2}J) + O_\pr(1) o_\pr(1)\}^{1/4} 
	= o_\pr(1). \nonumber
\end{align}
Here, we were able to apply ${J / \hat{\tau}_J(v)} = O_\pr(1)$,
due to the fact that
\begin{align*}
	\left| {\hat{\tau}_J(v) \over \tau_J(v)} - 1 \right|
	& = {1 \over J} {J \over \tau_J(v)} \langle (\widehat{\ga}_J^{-1}-\ga_J^{-1})v, v \rangle
	= O_\pr(n^{-1/2} J^3 (\log J)^2)^{1/2} J^{-1} O(1)
	\\& = O_\pr(n^{-1/2} J (\log J)^2)^{1/2}
\end{align*}
derived from \autoref{lemPerturb2}(b).
The bootstrap approximation can be derived in an analogous manner,
by replacing the term $\wsd^X(\sqrt{n/J} \widehat{\ga}_J^{-1/2}U_n, \s J^{-1/2} G_J)$ in \eqref{eq_thmProjRate}
with 
\begin{align*}
	& \wsd^X(\sqrt{n/J} \widehat{\ga}_J^{-1/2}U_n, \s J^{-1/2} G_J)
	+ \wsd^X(\sqrt{n/J} \widehat{\ga}_J^{-1/2}U_n^*, \sqrt{n/J} \widehat{\ga}_J^{-1/2}U_n)
	\\=& O_\pr(n^{-1/2}J) + O\left( \sum_{j>J} \g_j \langle \beta, \phi_j \rangle^2 \right),
\end{align*}
where the convergence rate is obtained by combining the rates in Equations~\eqref{eq_thmBTSrate} and~\eqref{eq_thmGau_rate}.
Since $J^{1/2} \sum_{j>J} \g_j \langle \beta, \phi_j \rangle^2 \leq C J^{1/2} \sum_{j>J} \langle \beta, \phi_j \rangle^2 = o(1)$ due to Condition~\ref{condProjSmooth}, 
the bootstrap result follows.

The proof will be complete if we show that the second term in \eqref{eq_thmProjBias} is negligible,
which can be proved by the following decomposition:
\begin{align}
	B_n
	= & \sqrt{n/J} \langle (\widehat{\Pi}_J-I)\beta, (\widehat{\Pi}_J-I)v \rangle \nonumber
	\\=& \sqrt{n/J} \langle (\widehat{\Pi}_J-\Pi_J)\beta, (\widehat{\Pi}_J-\Pi_J)v \rangle \label{eq_thmProj_bias1}
	\\& + \sqrt{n/J} \langle (\widehat{\Pi}_J-\Pi_J)\beta, (\Pi_J-I)v \rangle \label{eq_thmProj_bias2}
	\\& + \sqrt{n/J} \langle (\widehat{\Pi}_J-\Pi_J)v, (\Pi_J-I)\beta \rangle  \label{eq_thmProj_bias3}
	\\& + \sqrt{n/J} \langle (I-\Pi_J)\beta, (I-\Pi_J)v \rangle. \label{eq_thmProj_bias4}
\end{align}
The firs term \eqref{eq_thmProj_bias1} is bounded, by using \autoref{lemPerturb2}(a), as
\begin{align*}
	\sqrt{n/J} \langle (\widehat{\Pi}_J-\Pi_J)\beta, (\widehat{\Pi}_J-\Pi_J)v \rangle 
	& = O_\pr \left( \sqrt{n/J} (n^{-1/2} J^2 \log J)^2 \right)
	\\& = O_\pr \left( n^{-1/2} J^{7/2} (\log J)^2 \right),
\end{align*}
which is negligible by Condition~\ref{condProjRate}.
The second term \eqref{eq_thmProj_bias2} is bounded by
\begin{align*}
	& |\sqrt{n/J} \langle (\widehat{\Pi}_J-\Pi_J)\beta, (I-\Pi_J)v \rangle |
	\\ \leq & \sqrt{n/J} (n^{-1/2} J^2 \log J) \left( \sum_{j>J} \langle v, \phi_j \rangle^2 \right)^{1/2}
	\\=& \left( J^3 (\log J)^2 \sum_{j>J} \langle v, \phi_j \rangle^2 \right)^{1/2},
\end{align*}
using \autoref{lemPerturb2}(a).
Again using \autoref{lemPerturb2}(a),
the term in \eqref{eq_thmProj_bias3} is similarly bounded as
\begin{align*}
	|\sqrt{n/J} \langle (\widehat{\Pi}_J-\Pi_J)v, (I-\Pi_J)\beta \rangle| 
	= \left( J^3 (\log J)^2 \sum_{j>J} \langle \beta, \phi_j \rangle^2 \right)^{1/2}.
\end{align*}
%	
%	\tred{}
%	Since Condition~\ref{condProjRate} implies that $J^3 (\log J)^2 \leq C J^{3+\dt/2} \asymp \sqrt{n/J}$, 
%	using Condition~\ref{condProjSmooth}, 
%	we have that 
%	\begin{align*}
	%		& |\sqrt{n/J} \langle (I-\Pi_J)\beta, (I-\Pi_J)v \rangle|
	%		\\=& \left| \sqrt{n/J} \sum_{j>J} \langle \beta, \phi_j \rangle \langle v, \phi_j \rangle \right|
	%		\\ \asymp & \left( \sqrt{n \over J}\sum_{j>J} \langle \beta, \phi_j \rangle^2 \right)^{1/2}
	%		\left( \sqrt{n \over J}\sum_{j>J} \langle v, \phi_j \rangle^2 \right)^{1/2}.
	%	\end{align*}
%	
The fourth term \eqref{eq_thmProj_bias3} is bounded as
\begin{align*}
	& |\sqrt{n/J} \langle (I-\Pi_J)\beta, (I-\Pi_J)v \rangle|
	\\=& \left| \sqrt{n/J} \sum_{j>J} \langle \beta, \phi_j \rangle \langle v, \phi_j \rangle \right|
	\\ \leq & \left( \sqrt{n \over J}\sum_{j>J} \langle \beta, \phi_j \rangle^2 \right)^{1/2}
	\left( \sqrt{n \over J}\sum_{j>J} \langle v, \phi_j \rangle^2 \right)^{1/2}.
\end{align*}
Condition~\ref{condProjRate} implies that $J^3 (\log J)^2 \leq C J^{3+\dt/2} \asymp \sqrt{n/J}$, 
and using Condition~\ref{condProjSmooth}, 
we have that 
\begin{align*}
	\max \left\{ \sqrt{n/J}, J^3 (\log J)^2 \right\} \sum_{j>J} \langle \beta, \phi_j \rangle^2
	= J^{3+\dt/2}
	\sum_{j>J} \langle \beta, \phi_j \rangle^2 
	= o(1).
\end{align*}
The same argument can apply to $v$ instead of $\beta$,
which concludes the proof.
%	\tred{change Condition~\ref{condProjSmooth} with a multiplicative form}
%\end{proof}

\section{Functional calculus} \label{appTechBias}

Throughout this section, let $C$ denote a generic positive constant that can be different depending on cases.
%Also, for sequences $\{a_n\}$ and $\{b_n\}$ of positive real numbers, 
%$a_n \asymp b_n$ means that $a_n/b_n$ is asymptotically bounded away from both zero and infinity as $n\to\infty$. 
%Write $\Pi_J \equiv \sum_{j=1}^J \phi_j^{\otimes 2}$ and $\widehat{\Pi}_J \equiv \sum_{j=1}^J \hat{\phi}_j^{\otimes 2}$ for the projections operators onto the linear space spanned by $\{\phi_j\}_{j=1}^J$ and $\{\hat{\phi}_j\}_{j=1}^J$, respectively,
%and $\ga_J^{1/2} \equiv \sum_{j=1}^J \g_j^{1/2} \phi_j^{\otimes 2}$. 
%The bias term $B_n \equiv T_J - \sqrt{n/J} \widehat{\ga}_J^{-1/2} U_n$ in \eqref{eqDecompStat} is decomposed as
%\begin{align}
%	B_n
%	& \equiv \sqrt{n/J}  (\widehat{\ga}_J^{1/2} - \ga_J^{1/2}) (\widehat{\Pi}_J - \Pi_J) \beta \label{eqBias1}
%	\\& \hspace{12pt} + \sqrt{n/J}  \ga_J^{1/2} (\widehat{\Pi}_J - \Pi_J) \beta \label{eqBias2}
%	\\& \hspace{12pt} + \sqrt{n/J}  (\widehat{\ga}_J^{1/2}-\ga_J^{1/2}) (\Pi_J \beta - \beta)  \label{eqBias3}
%	\\& \hspace{12pt}  +  \sqrt{n/J}  \ga_J^{1/2} (\Pi_J \beta - \beta) \label{eqBias4}
%\end{align}
%To handle the terms in \eqref{eqBias1}–\eqref{eqBias4}, 
We employ standard techniques in functional calculus, following the approaches developed in \cite{CMS07} and \cite{YDN23RB}. 
Although our arguments draw on ideas from these works, 
the results we obtain do not follow directly from theirs.
Nevertheless, several technical steps resemble arguments already appearing in the literature;
therefore, we omit these repeated details and present only the points where our analysis departs from the existing literature.

Let $\bb_j \equiv \{ z \in \C:|z-\g_j| = \dt_j/2 \}$ be the oriented circle in the complex plane $\C$.
%and set $\cc \equiv \bigcup_{j=1}^J \bb_j$ to be their union.
with its sample counterpart $\widehat{\bb}_j \equiv \{ z \in \C:|z-\hat{\g}_j| = \hat{\dt}_j/2 \}$,
where $\hat{\dt}_1 \equiv \hat{\g}_1-\hat{\g}_2$ and $\hat{\dt}_j \equiv \min\{ \hat{\g}_j-\hat{\g}_{j-1}, \hat{\g}_{j+1}-\hat{\g}_j \}$ for $j \geq 2$. 
%and $\widehat{\cc}_n \equiv \bigcup_{j=1}^J \widehat{\bb}_j$. 
Then, using the resolvents $(zI - \ga)^{-1}$ and $(zI - \widehat{\ga})^{-1}$ of $\ga$ and $\widehat{\ga}$,  
we have contour integral expression for operators
as
\begin{alignat}{3}
	\Pi_J 
	& = {1 \over 2\pi \iota} \sum_{j=1}^J \int_{\bb_j} (zI-\ga)^{-1} dz,
	&& \qquad \widehat{\Pi}_J 
	&& = {1 \over 2\pi \iota} \sum_{j=1}^J \int_{\widehat{\bb}_j} (zI-\widehat{\ga})^{-1} dz, \\
	\ga_J^{1/2} 
	& = {1 \over 2\pi \iota} \sum_{j=1}^J \int_{\bb_j} z^{1/2} (zI-\ga)^{-1} dz,
	&& \qquad \widehat{\ga}_J ^{1/2}
	&& = {1 \over 2\pi \iota} \sum_{j=1}^J \int_{\widehat{\bb}_j}  z^{1/2} (zI - \widehat{\ga})^{-1} dz,
	%	\ga_J^{-1} 
	%	& = {1 \over 2\pi \iota} \sum_{j=1}^J \int_{\bb_j} z^{-1} (zI-\ga)^{-1} dz,
	%	&& \qquad \widehat{\ga}_J^{-1}
	%	&& = {1 \over 2\pi \iota} \sum_{j=1}^J \int_{\widehat{\bb}_j} z^{-1} (zI-\widehat{\ga})^{-1} dz,
\end{alignat}
where $\iota^2 = -1$ \citep[Chapter~VII]{DS88}.
We also introduce other notation related to perturbation theory for later development:
\begin{align*}
	G(z) 
	& \equiv (zI - \ga)^{-1/2} (\widehat{\ga}-\ga)(zI - \ga)^{-1/2}, \\
	K(z)
	& \equiv (zI - \ga)^{1/2} (zI - \widehat{\ga})^{-1} (zI - \ga)^{1/2}, \\
	\ee_j 
	& \equiv ( \opnorm{G(z)}_\infty < 1/2, \forall z \in \bb_j), \\
	\aaa_J,
	& \equiv ( \forall j \in \{1, \dots, J\}, |\hat{\g}_j-\hat{\g}_j| < \dt_j/2 ). 
\end{align*}

We will frequently use the following lemmas, the results that come from \cite{CMS07} and \cite{YDN23RB}, without referring to them each time.
\begin{lem} \label{lemCMS07}
	Suppose that  Conditions~\ref{condMomentX}--\ref{condConvexEV} hold.
	Then, we have the following:
	\begin{enumerate}[(a)]
		
		\item 
		$\sup_{z \in \bb_j} \sum_{l=1}^\infty {\g_l / |z-\g_l|} \leq C j \log j$;

		\item 
		$\eo \left[ \sup_{z \in \bb_j} \opnorm{G(z)}_\infty ^2 \right] \leq C n^{-1} (j \log j)^2$;
		
		\item 
		$\sup_{z \in \bb_j} \opnorm{K(z)}_\infty  \I_{\ee_j} \leq C$ almost surely;
		
		\item 
		$\pr(\ee_j^c) \leq C n^{-1/2} j\log j$; and
		
		\item 
		$\pr(\aaa_J^c) \leq C \left( n^{-1} \sum_{j=1}^J \dt_j^{-2} + n^{-1/2} \sum_{j=1}^J j \log j \right)$.
		
	\end{enumerate}
	In this case, parts~(d)--(e) imply that, for any random quantities $\{q_j\}_{j=1}^J$ and $r_n$ and for each $\eta>0$, 
	\begin{align*}
		\pr \left( \sum_{j=1}^J q_j \I_{\ee_j^c} >\eta \right)
		&  \leq C n^{-1/2} \sum_{j=1}^J j \log j, \\
		\pr(r_n\I_{\aaa_J^c}>\eta)
		& \leq C \left( n^{-1} \sum_{j=1}^J \dt_j^{-2} + n^{-1/2} \sum_{j=1}^J j \log j \right).
	\end{align*}
\end{lem}

\begin{lem} \label{lemPerturb1}
	Suppose that Conditions~\ref{condMomentX}--\ref{condBasicTrunc} hold. 
	Then, we have
	\begin{align*}
		\|\sqrt{n/J} \ga^{1/2} (\widehat{\Pi}_J - \Pi_J) \beta\|^2
		= D_{1n}  + D_{2n}
	\end{align*}
	where $\eo[D_{1n}] = o(1) + O(n^{-1/2} J^{5/2} (\log J)^2)$ and $r_nD_{2n} = o_\pr(1)$ for any random quantity $r_n$. 
\end{lem}

\begin{proof}
	Following the proof of \citet[Proposition~2]{CMS07}, 
	$\widehat{\Pi}_J - \Pi_J$ is decomposed as
	\begin{align*}
		\widehat{\Pi}_J - \Pi_J
		\equiv \sss_J + \rr_J + r_n\I_{\aaa_J^c},
	\end{align*}
	where
	\begin{align*}
		\sss_J 
		& \equiv {1 \over 2\pi\iota} \sum_{j=1}^J \int_{\bb_j} (zI - \ga^{-1}) (\widehat{\ga}-\ga) (zI - \ga)^{-1} dz, \\
		\rr_J
		& \equiv {1 \over 2\pi\iota} \sum_{j=1}^J \int_{\bb_j} (zI - \ga)^{-1} (\widehat{\ga}-\ga) (zI - \ga)^{-1} (\widehat{\ga}-\ga) (zI - \widehat{\ga})^{-1} dz,
	\end{align*}
	and $r_n$ is some random quantity.
	Then,
	\begin{align}
		& \|\sqrt{n/J} \ga^{1/2} (\widehat{\Pi}_J - \Pi_J) \beta\| \nonumber
		\\ \leq & \|\sqrt{n/J} \ga^{1/2} \sss_J \beta\|
		+ \|\sqrt{n/J} \ga^{1/2} \rr_J \beta\|
		+ \|\sqrt{n/J} \ga^{1/2} r_n \beta\| \I_{\aaa_J^c} \label{eqLemBias1decomp}
	\end{align}

	Following the argument in \citet[pages~347--350]{CMS07},
	we derive that 
	\begin{align} \label{eqLemBias1sss}
		\eo[\|\sqrt{n/J} \ga^{1/2} \sss_J \beta\|^2]
		= {n \over J} \sum_{j=1}^\infty  \g_j \eo[\langle \sss_J \beta, \phi_j \rangle^2] = o(1). 
	\end{align}
	We note that the assumption that $\sum_{j=1}^\infty |\langle \beta, \phi_j \rangle|<\infty$ imposed by \cite{CMS07} and \cite{YDN23RB} is not necessary for proving \eqref{eqLemBias1sss}, 
	as shown by \cite{Yeon26} (cf.~Section~S3.1 of its supplement). 
	Since it is straightforward to remove the condition that $\sum_{j=1}^\infty |\langle \beta, \phi_j \rangle|<\infty$, 
	we do not provide the details behind it. 
	
	For the second term in \eqref{eqLemBias1decomp}, 
	note that
	\begin{align*}
		\sqrt{n \over J} \ga^{1/2} \rr_n\beta = {\sqrt{n/J} \over 2\pi\iota} \sum_{j=1}^J \int_{\bb_j} \ga^{1/2} (zI-\ga)^{-1/2} G_n(z)^2 K_n(z) (zI-\ga)^{-1/2}\beta dz.
	\end{align*}
	Using \autoref{lemCMS07}(a), we first observe 
	\begin{align*}
		\sup_{z \in \bb_j} \opnorm{\ga^{1/2} (zI-\ga)^{-1/2}}_2^2 = \sum_{l=1}^\infty {\g_l \over |z-\g_l|}
		\leq C j \log j.
	\end{align*}
	Then, by following the argument in \citet[pages~351--532]{CMS07} for deriving Equation~(28) there, 
	we have $\|\sqrt{n/J}\ga_J^{1/2} \rr_n\beta\| \leq \sum_{j=1}^J A_j I_{\ee_j} + \sum_{j=1}^J A_j \I_{\ee_j^c}$,
	where
	\begin{align*}
		\eo \left[ \sum_{j=1}^J A_j I_{\ee_j} \right]
		& \leq C \sum_{j=1}^J \dt_j (j \log j)^{1/2} (n^{-1/2} j \log j)^2  \dt_j^{-1/2}
		\\& \leq C n^{-1/2} \sum_{j=1}^J (j \log j)^2
		\leq n^{-1/2} J^{5/2} (\log J)^2. 
	\end{align*}
	Since $\sum_{j=1}^J A_j \I_{\ee_j^c} = o_\pr(1)$ and $\|\sqrt{n/J} \ga^{1/2} r_n \beta\| \I_{\aaa_J^c} = o_\pr(1)$ by \autoref{lemCMS07}(d)--(e), 
	we have the desired result.
\end{proof}

\begin{lem}\label{lemPerturb2}
	Suppose that  Conditions~\ref{condMomentX}--\ref{condBasicTrunc} hold.
	Then, we have the following:
	\begin{enumerate}[(a)]
		%		\item 
		%		$\opnorm{\widehat{\ga}_J^{1/2} - \ga_J^{1/2}}_\infty = D_{1n} + D_{2n}$, where $\eo[D_{1n}] = O(n^{-1/2} J^{3/2} (\log J)^{1/2})$ and $r_nD_{2n} = o_\pr(1)$ for any random quantity $r_n$. 
		
		\item 
		$\opnorm{\widehat{\Pi}_J - \Pi_J}_\infty = D_{1n} + D_{2n}$, where $\eo[D_{1n}] = O(n^{-1/2} J^2 \log J)$ and $r_nD_{2n} = o_\pr(1)$ for any random quantity $r_n$. 
		
		\item
		$\opnorm{\widehat{\ga}_J^{-1} - \ga_J^{-1}}_\infty = D_{1n} + D_{2n}$, where $\eo[D_{1n}] = o(n^{-1/2} J^3 (\log J)^2)^{1/2}$ and $r_nD_{2n} = o_\pr(1)$ for any random quantity $r_n$.

	\end{enumerate}
\end{lem}

For brevity, we omit the proof of \autoref{lemPerturb2}, as it follows from the standard argument same as that for deriving the convergence of the second term in \eqref{eqLemBias1decomp}. 
See \citet[Lemma~S3 of the supplement]{YDN23RB}, for example.

\section{Ancillary lemmas} \label{appTechLem}

\begin{lem} \label{lemAlge}
	For $x,y \in \HH$ and a bounded linear operator $A$ on $\HH$, we have the following:
	\begin{enumerate}[(a)]
		\item 
		$\opnorm{x\otimes y}_1 = \|x\|\|y\|$;
		
		\item 
		$(Ax)^{\otimes 2} = Ax^{\otimes 2} A^\top$,
		where $A^\top$ stands for the adjoint operator of $A$.
	\end{enumerate}
\end{lem}

\begin{proof}
	These follow from  \citet[Lemmas~S2(c) and~S1(d)]{Yeon26}.
\end{proof}

\begin{lem} \label{lemWSDcont}
	Let $U,V$ be random elements taking values in a Banach space $\B$ with norm $\|\cdot\|$,
	and $f:\B\to\B'$ be a Lipschitz continuous function on $\B$ with constant $K \in (0,\infty)$ with a potentially distinct Banach space $\B'$ with norm $\|\cdot\|'$.
	Then, we have 
	\begin{align*}
		\wsd(f(U), f(V)) \leq K \cdot \wsd(U, V).
	\end{align*}
\end{lem}

\begin{proof}
	Since the infimum in the Wasserstein distance is achieved \cite[cf.][Lemma~8.1]{Mallows72},
	we take random elements $\check{U}$ and $\check{V}$ taking values in $\B$ such that $\check{U} \ed U$, $\check{V} \ed V$ and $\wsd(U,V)^2 = \eo[\|\check{U}-\check{V}\|^2]$.
	Then, using the Lipschitz continuity, we have the desired result, as
	\begin{align*}
		\wsd(f(U), f(V))
		& = \inf_{U^\dagger\ed U, V^\dagger\ed V} \eo[\|f(U^\dagger)-f(V^\dagger)\|'{^2}]
		\\& \leq K^2 \inf_{U^\dagger\ed U, V^\dagger\ed V} \eo[\|U^\dagger-V^\dagger\|{^2}]
		\\&  \leq K^2 \eo[\|\check{U}-\check{V}\|^2]
		= K^2 \wsd(U,V)^2.
	\end{align*}
\end{proof}

\begin{lem} \label{lemSobolevEmb}
	For $f \in \W^{1,2}([0,1])$, we have
	\begin{align*}
		\sup_{u \in [0,1]}|f(u)| \leq \sqrt{2} \|f\|_{\W^{1,2}([0,1])},
	\end{align*}
	where $\|\cdot\|_{\W^{1,2}([0,1])}$ denotes the norm on $\W^{1,2}([0,1])$ defined by  
	\begin{align*}
		\|f\|_{\W^{1,2}([0,1])}^2 \equiv \|f\|_{L^2([0,1])}^2 + \|f'\|_{L^2([0,1])}^2,
		\quad f \in \W^{1,2}([0,1]),
	\end{align*}
	and $\|\cdot\|_{L^2([0,1])}$ is the norm on $L^2([0,1])$ defined by $\|f\|_{L^2([0,1])}^2 \equiv \int_0^1 f(u)^2 du$ for $f \in L^2([0,1])$.
\end{lem}

\begin{proof}
	We follow the line of the proof of an Sobolev embedding theorem such as those in \citet[Theorem~4 in Chapter~5]{Evans10}. 
	By the fundamental theorem of calculus, for $u,v \in [0,1]$, we have
	\begin{align*}
		|f(u)-f(v)| = \left| \int_v^u f'(s)ds \right|
		\leq \int_v^u |f'(s)|ds,
	\end{align*}
	implying
	\begin{align*}
		|f(u)| \leq |f(u)-f(v)| + |f(v)|
		\leq \int_v^u |f'(s)| ds + |f(v)|.
	\end{align*}
	Then, for each $u \in [0,1]$, 
	\begin{align*}
		|f(u)|
		& = \int_0^1 |f(u)| dv
		\leq \int_0^1 \int_v^u |f'(s)| ds dv + \int_0^1 |f(v)| dv
		\\& = \int_0^u \int_0^s |f'(s)| dv ds + \int_0^1 |f(v)| dv
		\\& = \int_0^u s|f'(s)| ds + \int_0^1 |f(v)| dv
		\\& \leq  \int_0^1 |f'(s)|ds + \int_0^1 |f(v)| dv.
	\end{align*}
	This concludes that
	\begin{align*}
		\sup_{u \in [0,1]} |f(u)|
		&\leq \int_0^1 |f'(s)|ds + \int_0^1 |f(v)| dv
		\\& \leq \left( \int_0^1 f'(s)^2ds \right)^{1/2} + \left( \int_0^1 f(v)^2 dv \right)^{1/2}
		\\& = \|f\|_{L^2([0,1])} + \|f'\|_{L^2([0,1])}
		\\& \leq \sqrt{2} \left( \|f\|_{L^2([0,1])}^2 + \|f'\|_{L^2([0,1])}^2 \right)^{1/2}
		= \sqrt{2} \|f\|_{\W^{1,2}([0,1])}.
	\end{align*}
	
\end{proof}

\begin{lem} \label{lemKDbddWD}
	Let $F,G$ be the cumulative distribution functions with probability distributions $P,Q$ on $\R$, respectively.
	If $G$ is absolutely continuous with $\sup_{s \in \R} |G'(s)|\leq C$,
	then 
	\begin{align*}
		\sup_{s \in \R} |F(s)-G(s)| \leq \sqrt{2C \wsd(P,Q)}. 
	\end{align*}
\end{lem}

\begin{proof}
	See \citet[Prposition~1.2]{Ross11}, for example.
\end{proof}

\begin{lem} \label{lemSqDensity}
	Let $G_J$ be a Gaussian random element taking values in $\HH$ with mean zero and covariance operator $\Pi_J \equiv \sum_{j=1}^J \phi_j^{\otimes 2}$ for some orthonormal basis elements $\{\phi_j\}_{j=1}^\infty$ in $\HH$.
	\begin{enumerate}[(a)]
		\item 
		The distribution of $\|G_J\|^2$ is continuous.
		
		\item 
		The density of  $\|G_J\|^2$  is bounded by $CJ^{-1/2}$ for sufficiently large $J$, where $C$ is a generic constant.
		
	\end{enumerate}
\end{lem}

\begin{proof}
	The Karhunen--Lo\'{e}ve expansion provides $G_J = \sum_{j=1}^J \xi_j \phi_j$, where $\{\xi_j\}_{j=1}^J$ are iid standard normal variables.
	This then implies that $\|G_J\|^2 = \sum_{j=1}^J \xi_j^2 \sim \chi^2(J)$, whose distribution is continuous.
	For the second part, it is well known that the density of $\chi^2(J)$ is maximized at $J-2$ if $J > 2$. 
	The maximum value is then 
	\begin{align*}
		M_J \equiv {1 \over 2^{J/2} \ga(J/2)} (J-2)^{J/2-1} e^{-(J-2)/2}.
	\end{align*}
	Using the Sterling's formula \cite[e.g.,][Formula~25.15]{SLL99},
	as $J\to\infty$, we have
	\begin{align*}
		\ga(J/2) = {\sqrt{2\pi} \over (J/2)^{1/2}} {(J/2)^{J/2} \over e^{J/2}} \{1 + O(2/J)\},
	\end{align*}
	yielding
	\begin{align*}
		M_J
		= {e \over 2 \sqrt{\pi}} {J^{1/2} \over J-2} (1-2/J)^{J/2} \{1 + O(2/J)\}.
	\end{align*}
	Since $\lim_{J\to\infty}  (1-2/J)^{J/2} = e^{-1}$, we conclude $M_J = O(J^{-1/2})$.
\end{proof}

\begin{lem} \label{lemSupDensity}
	Let $\{G_J(u)\}_{u \in [0,1]}$ be a Gaussian process with mean zero and covariance function $\Pi_J$ on $[0,1]^2$ defined by $\Pi_J(u_1,u_2) \equiv \sum_{j=1}^J \phi_j(u_1)\phi_j(u_2)$ for $(u_1,u_2)^\top \in [0,1]^2$,
	where $\{\phi_j\}_{j=1}^\infty$ satisfies $\int_0^1 \phi_j(u) \phi_j'(u) du = \I(j=j')$ for each $j,j' \in \N$. 
	\begin{enumerate}[(a)]
		\item The distribution of $\sup_{u \in [0,1]}|G_J(u)|$ is continuous.
		
		\item The density of $\sup_{u \in [0,1]}|G_J(u)|$ is bounded by $C J^{1/4}$, 
		where $C$ is a generic constant.
	\end{enumerate}
\end{lem}

\begin{proof}
	The Karhunen--Lo\'{e}ve expansion provides $G_J(u) = \sum_{j=1}^J \xi_j \phi_j(u)$ for $t \in [0,1]$, where $\{\xi_j\}_{j=1}^J$ are iid standard normal variables.
	Then, by the continuity of $\phi_j$ \cite[cf.][Section~7.3]{HE15} and the compactness of $[0,1]$,
	we have
	\begin{align*}
		\sup_{u \in [0,1]} |G_J(u)| \leq \sum_{j=1}^J |\xi_j| \sup_{u \in [0,1]} |\phi_j(u)| < \infty.
	\end{align*}
	Therefore, the distribution of $\sup_{u \in [0,1]}|G_J(u)|$ is continuous \cite[e.g.,][Theorems~1 or~3]{Yang25}.
	
	To prove the second part, 
	let us introduce some notation along with several well-known facts.
	Define $\bm{\phi}:[0,1]\to\R^J$ as $\bm{\phi}(u) \equiv (\phi_1(u), \dots, \phi_J(u))^\top$ for $u \in [0,1]$.
	We then construct a norm $\|\cdot\|_\phi$ on $\R^J$ as $\|\bm{v}\|_\phi \equiv \sup_{u \in [0,1]} |\bm{v}^\top \bm{\phi}(u)|$ for $\bm{v} \in \R^J$ and denote the Euclidean norm on $\R^J$ as $\|\cdot\|_{\R^J}$. 
	Since any norms on $\R^J$ are equivalent,
	there exist constants $C_1>0,C_2 > 0$ such that $C_1 \|\bm{v}\|_{\R^J} \leq \|\bm{v}\|_\phi \leq C_2 \|\bm{v}\|_{\R^J}$ for all $\bm{v} \in \R^J$. 
	For $r \geq 0$, we write the closed ball based on the $\phi$-norm $\|\cdot\|_\phi$ as $B_{\phi,r}$
	\begin{align*}
		%		B_r 
		%		& \equiv \{ \bm{v} \in \R^J:\|\bm{v}\| \leq r \}, \quad
		B_{\phi,r}
		\equiv \{ \bm{v} \in \R^J: \|\bm{v}\|_\phi \leq r \}.
	\end{align*}
	Since $\|\cdot\|_\phi$ is a norm, the $\phi$-ball $B_{\phi,r}$ is convex in $\R^J$. 
	For a set $A \subseteq \R^J$ and $\e \in (0,\infty)$, we define its $\e$-enlargements in either norms as 
	\begin{align*}
		A^\e 
		& \equiv \{ \bm{x} \in \R^J: \inf\{ \|\bm{x}-\bm{y}\|:\bm{y} \in A \} \leq \e \}, \\
		A^{\phi,\e} 
		& \equiv \{ \bm{x} \in \R^J: \inf\{ \|\bm{x}-\bm{y}\|:\bm{y} \in A \} \leq \e \},
	\end{align*}
	respectively.
	Due to the equivalence between the norms, for each $r \geq 0$ and $\e \in (0,\infty)$, we have $(B_{\phi,r})^{\phi,\e} \subseteq (B_{\phi,r})^{\e/C_1}$. 
	
	We now set $\bm{\xi} \equiv (\xi_1, \dots, \xi_J)^\top$ so that $\bm{\xi} \sim \nd(\bm{0},\bm{I}_J)$ is a $J$-dimensional standard Gaussian random vector and
	\begin{align*}
		M_J \equiv \sup_{u \in [0,1]}|G_J(u)| 
		= \sup_{u \in [0,1]} |\bm{\xi}^\top \bm{\phi}(u)| = \|\bm{\xi}\|_\phi.
	\end{align*}
	Then, by using the upper bounds of the Gaussian probability over the boundary of a convex set in $\R^J$ as in \cite{Ball93} or \cite{Raic19},
	we derive
	\begin{align*}
		& \sup \left\{ \e^{-1} \{\pr (M_J \leq r+\e) - \pr(M_J \leq r)\} : \e \in (0,\infty) \right\}
		\\ = & \sup \left\{ \e^{-1}\pr(\bm{\xi} \in (B_{\phi,r})^{\phi,\e} \bsh B_{\phi,r}):\e \in (0,\infty) \right\}
		\\ \leq & C_1^{-1} \sup \left\{ C_1\e^{-1}\pr(\bm{\xi} \in (B_{\phi,r})^{\e/C_1} \bsh B_{\phi,r}):\e \in (0,\infty) \right\}
		\\ \leq & CJ^{1/4},
	\end{align*}
	where $C$ denotes a generic constant.
	This proves the second part and completes the proof. 
\end{proof}

\begin{lem} \label{egCounterCond3}
	Let $\{a_j\}$ be a non-increasing sequence of positive real values such that $\sum_{j=1}^\infty a_j<\infty$.
	Then, there is a case where $\sup_{j \in \N} a_j j \log j = \infty$. 
\end{lem}

\begin{proof}
	For this counterexample, we consider a sequence defined through dyadic blocks.
	Set $a_1=1$, and for each $j \in \N \bsh \{1\}$, pick $m \in \N$ such that $2^m \leq j < 2^{m+1}$ and $k\in \N$ such that $(k-1)^4 < m \leq k^4$,
	and define 
	\begin{align*}
		a_j = {1 \over k^2 2^{k^4}}.
	\end{align*}
	Note that the $a_j$'s are positive and the sequence is non-increasing. 
	Also, this sequence is summable as
	\begin{align*}
		\sum_{j=1}^\infty a_j
		& = 1 + \sum_{k=1}^\infty \sum_{m>(k-1)^4}^{k^4} \sum_{j=2^m}^{2^{m+1}-1} {1 \over k^2 2^{k^4}}
		= 1 + \sum_{k=1}^\infty \sum_{m>(k-1)^4}^{k^4} 2^m {1 \over k^2 2^{k^4}}
		\\& \leq 1 + \sum_{k=1}^\infty {2 \over k^2} <\infty
	\end{align*}
	since $\sum_{m>(k-1)^4}^{k^4} 2^m
	= 2^{k^4+1}-2^{(k-1)^4}
	< 2^{k^4+1}$.
	However, $a_j j \log j$ is not bounded 
	because
	\begin{align*}
		\sup_{j \in \N} a_j j \log j
		& \geq \sup_{k \in \N} {2^{k^4} \log (2^{k^4}) \over k^2 2^{k^4}}
		= \sup_{k \in \N} k^2 \log 2 = \infty. 
	\end{align*}
	
\end{proof}

\begin{lem} \label{lemProjIndicator}
	Suppose that the eigenfunctions are given by the trigonometric functions in \eqref{eqBasisTri} that forms an orthonormal basis for $\HH = L^2([0,1])$ \citep[Theorem~2.4.18]{HE15} and that $\g_j \asymp j^{-a}$ for some $a>2$.
	We further assume that the projection direction of interest is given by an indicator as $v(u) = \I(a \leq u \leq b)$ for some $a,b \in [0,1]$ with $a<b$ and $l \equiv b-a < 1$; clearly, $v \in \HH$. 
	Then, Condition~\ref{condProjScale} is satisfied as 
	\begin{align*}
		{J \over \tau_J(v)} \asymp J^{-(a-2)} = o(1).
	\end{align*}
\end{lem}

\begin{proof}
	We first note that
	\begin{align*}
		\langle v, \hat{\phi}_{\mathrm{tri},2m} \rangle
		& = \int_a^b \sqrt{2} \sin (2m\pi u) du
		= {1 \over \sqrt{2} m \pi} \{\cos (2m\pi a) - \cos(2m\pi b)\}, \\
		\langle v, \hat{\phi}_{\mathrm{tri},2m+1} \rangle
		& \int_a^b \sqrt{2} \cos (2m \pi u) du
		= {1 \over \sqrt{2} m \pi} \{\sin (2m\pi b) - \sin(2m\pi a)\}, 
	\end{align*}
	implying
	\begin{align*}
		& \langle v, \hat{\phi}_{\mathrm{tri},2m} \rangle^2
		+ \langle v, \hat{\phi}_{\mathrm{tri},2m+1} \rangle^2
		\\& = {1 \over 2\pi^2 m^2} [2 - 2 \{\cos(2m\pi a) \cos(2m\pi b) + \sin(2m\pi a) \sin(2m\pi a)\}]
		\\& = {1 \over \pi^2 m^2} \{1-\cos(2m\pi l)\}.
		= {2 \over \pi^2 m^2} \sin^2(m\pi l).
	\end{align*}
	This leads us to studying the sum $\ka^p(M) \equiv \sum_{m=1}^M m^p \sin^2(m\pi l)$
	for a real number $p>-1$.
	We claim that
	\begin{align} \label{eq_lemProjIndicator}
		\ka^p(M) \asymp M^{p+1},
	\end{align} 
	where its derivation is deferred later. 
	Then, the claim in \eqref{eq_lemProjIndicator} implies that  $\tau_J(v) \asymp J^{a-1}$, 
	where the latter derives the desired result.

	For any real numbers $x$ and $y$, 
	using various trigonometric identities,
	we observe that
	\begin{align*}
		\sin^2(x) + \sin^2(x+y)
		& = {1- \cos(2x) \over 2} + {1-\cos(2x+2y) \over 2}
		= 1-\cos(2x+y)\cos(y) 
		\\& \geq 1-|\cos(y)|,
	\end{align*}
	where the last inequality simply comes from the fact that $|\cos(2x+y)| \leq 1$. 
	For $r \in \N$, 
	by putting $x = \pi(2r-1)l$ and $y = \pi l$,
	we obtain
	\begin{align*}
		\sin^2(\pi(2r-1)l) + \sin^2(2\pi r l) \geq c_l \equiv 1-|\cos(\pi l)| > 0. 
	\end{align*}
	This yields a lower bound for $\ka(M)$ as
	\begin{align*}
		\ka(M)
		& \geq \sum_{r=1}^{\lfloor M/2 \rfloor} \left\{ (2r-1)^p \sin^2 (\pi(2r-1)l) + (2r)^p \sin^2(2\pi rl) \right\}
		\\& \geq c_l \sum_{r=1}^{\lfloor M/2 \rfloor} (2r-1)^p
		\asymp M^p+1. 
	\end{align*}
	Therefore, the claim in \eqref{eq_lemProjIndicator} is proved,
	because $\ka^p(M) \leq \sum_{m=1}^M m^p \asymp M^{p+1}$.
\end{proof}

\section{Additional discussion about theoretical results}
\label{secAddDiscuss}

In light of comments by an anonymous reviewer, 
we expand our discussion below on several key theoretical points, 
which help to deepen our insights into the problem.

\subsection{Relationship with prediction error}
\label{ssecPE}

A substantive point is that our result changes the usual perspective on FPCR-based functional regression inference, 
and our operator scaling is clearly distinct from the role of square-root covariance operator appearing in prediction error.

To clarify this distinction, consider prediction under the SoFR framework. Let $X_0$ be a new predictor and
$Y_0 = \mu(X_0) + \e_0$ 
be a new response at $X_0$
where the error $\e_0$ independent of the data $\dd_n \equiv \{(X_i,Y_i)\}_{i=1}^n$.
The (conditional) mean function is 
\begin{align*}
	\mu(x) 
	\equiv \eo[Y|X=x] 
	= \ap + \langle \beta, x \rangle
	= \eo[Y] + \langle \beta, x-\eo[X] \rangle.
\end{align*}
Given an estimator $\hat{\beta}$ for $\beta$,
a natural estimated mean function $\hat{\mu}$ is defined by $\hat{\mu}(x) = \bar{Y} + \hat{\beta}(x-\bar{X})$,
which produces the predicted value $\hat{Y}_0 \equiv \hat{\mu}(X_0)$ at $X_0$.
Conditional on the data $\dd_n$, the prediction error satisfies
\begin{align*}
	\eo[(\hat{Y}_0-Y_0)^2|\dd_n]
	& = \eo[
	\{(\bar{Y} + \langle \hat{\beta}, X_0-\bar{X}\rangle) 
	- (\eo[Y] + \langle \beta, X_0-\eo[X] \rangle + \e_0)\}^2
	|\dd_n]
	\\& = \eo [
	(\bar{Y}-\eo[Y] + \langle \hat{\beta}-\beta, X_0-\eo[X] \rangle - \langle \beta, \bar{X}-\eo[X] \rangle + \e_0)^2
	|\dd_n ]
	\\& \leq 4(\bar{Y}-\eo[Y])^2 
	+ 4\langle \beta, \bar{X}-\eo[X] \rangle^2
	+ 4\e_0^2
	+ 4\eo[\langle \hat{\beta}-\beta, X_0-\eo[X] \rangle^2|\dd_n].
\end{align*}
Thus, apart from the noise and mean-estimation terms,
the slope-estimation contribution to prediction is governed by
\begin{align*}
	\eo[\langle \hat{\beta}-\beta, X_0-\eo[X] \rangle^2|\dd_n]
	& = \eo[\langle (X_0-\eo[X])^{\otimes 2} (\hat{\beta}-\beta), \hat{\beta}-\beta \rangle|\dd_n]
	\\& = \langle \eo[(X_0 - \eo[X])^{\otimes 2}] (\hat{\beta}-\beta), \hat{\beta}-\beta \rangle.
\end{align*}
When $X_0$ has the same distribution as the observed predictors,
we have $\var[X_0] = \eo[(X_0 - \eo[X])^{\otimes 2}] = \ga,$
and hence this term becomes 
\begin{align*}
	\eo[\langle \hat{\beta}-\beta, X_0-\eo[X] \rangle^2|\dd_n]
	& = \langle \ga (\hat{\beta}-\beta), \hat{\beta}-\beta \rangle
	\\& = \|\ga^{1/2}(\hat{\beta}-\beta)\|^2
\end{align*}
Technically, prediction does not use ``operator scaling'' as a deliberate device for constructing the predictor. 
Rather, the operator $\ga^{1/2}$ appears because the prediction error naturally induces the $\ga^{1/2}$-weighted norm. 
This is different from our use of the scaling $\widehat{\ga}_J^{1/2}$ for inference,
which not only stabilizes the limiting behavior of the statistic but also completely removes the bias terms from the statistic. 
Moreover, if $X_0$ is not identically distributed as $X_1$, 
then the norm induced by the prediction error involves $\var[X_0]^{1/2}$, 
which may differ from $\ga^{1/2}$.

By contrast, inference necessarily involves centering and scaling. 
As the simplest example, the sample mean has a non-degenerate limiting distribution only after scaling by $\sqrt{n}$.
In more complex problems, the appropriate scaling can be slower than $\sqrt{n}$, such as $\sqrt{nh}$ with bandwidth $h\to 0$ in local linear regression, or $n^{1/3}$ in shape-constrained problems. 
For the FPCR estimator $\hat{\beta}_J$, however, one does not generally obtain a non-degenerate CLT under any scalar scaling \citep{CMS07}. 
To obtain meaningful distributional approximations for $\hat{\beta}_J$, 
we therefore introduce the operator scaling $\widehat{\ga}_J^{1/2}$. 
As described in the Introduction, this operator scaling leads to simple and effective inference tools for SoFR 
and can be extended to other regression problems where FPCR-type estimators are used.
This is distinct from prediction, 
where $\ga^{1/2}$ arises naturally from the prediction error itself, rather than being introduced as a device to improve prediction.

Numerically, these two statistical tasks need not be aligned.
For instance, one may choose an estimator or tuning rule that minimizes prediction error, for example through cross-validation. 
However, our numerical studies suggest that such a choice does not necessarily lead to reliable inference in the settings considered here. 
At the same time, prediction error still appears in the theoretical analysis of bootstrap consistency: 
as shown in Appendix~\ref{appTechMain1}, the Wasserstein distance between the bootstrap and true error distributions is controlled by a term involving the prediction error. 
Thus, while prediction performance and inferential validity may differ substantially in practice, prediction error naturally enters the argument for the proposed bootstrap approximation.

\subsection{Inference stability under exponential decay rates}
\label{ssecExpDecay}

Suppose that the eigen-gaps decay exponentially, say $\dt_j \asymp e^{-dj}$, 
as is the case when exponentially decaying eigenvalues are sufficiently well separated. 
Then the last term in Condition~(C4) becomes
\begin{align*}
	n^{-1}\sum_{j=1}^J \dt_j^{-2}
	\asymp
	n^{-1}\sum_{j=1}^J e^{2dj}
	\asymp
	n^{-1}e^{2dJ}.
\end{align*}
Thus, Condition~(C4) requires
\[
e^{2dJ}/n=o(1),
\]
or equivalently $2dJ-\log n\to -\infty$. 
This illustrates that exponential decay imposes a stringent stability requirement on the truncation level, because the inverse-eigen-gap factors grow exponentially fast.

Recall also that the empirical prediction error appearing in the upper bound  $\mathsf{W}^X(Q, \widehat{Q})$ for the Wasserstein distance from the bootstrap statistic satisfies
\begin{align*}
	\|\hat{\ga}^{1/2}(\hat{\beta}_J-\beta)\|^2
	= O_\pr(n^{-1/2})
	+ O_\pr(J/n)
	+ O\left(\sum_{j>J}\g_j\langle \beta,\phi_j\rangle^2\right)
\end{align*}
(cf.~\autoref{appTechMain1}). 
Under the above stability requirement, the variance term $J/n$ is negligible; in particular, Condition~(C4) implies that $J/n=o(1)$. 
However, the bias term can still be sensitive to the choice of $J$. 
For instance, if the tail term behaves like $e^{-J}$, then minimizing the unconstrained upper bound $J/n+e^{-J}$ suggests a choice of order $J\asymp \log n$. 
Such a choice may be incompatible with Condition~(C4): for the canonical choice $J \asymp \log n$, we have
\[
n^{-1}e^{2dJ} \asymp n^{2d-1},
\]
which is not negligible when $d\ge 1/2$. 
More generally, if the tail term behaves like $e^{-\eta J}$, then the unconstrained choice is $J\asymp \eta^{-1}\log n$, and compatibility with Condition~(C4) requires $\eta>2d$, up to constants. 
Therefore, when the eigen-gaps decay exponentially, Condition~(C4) can impose a stronger upper bound on the admissible truncation level $J$ 
than the one suggested by minimizing the prediction-error upper bound. 
Thus, bootstrap validity may require choosing $J$ with inference stability in mind, 
rather than using the unconstrained prediction-error-optimal choice. 
When the coefficients of $\beta$ decay sufficiently rapidly, this restriction may be less consequential, 
since the bias term can already be small for a more conservative choice of $J$.

\subsection{Further discussion of projection inference} \label{ssecFurtherDiscussionProjInfer}

We briefly discuss the existing literature on projection inference in functional regression. 
To our knowledge, the only papers in this direction are \cite{YDN23RB, CMS07, GM11, KH16a}. 
First, all of these works assume that $X$ and $Y$ have mean zero.
Moreover, the latter three do not account for centering by the sample means $\bar{X}$ and $\bar{Y}$, 
which may limit the practical relevance of their theoretical results. 
The results in \cite{YDN23RB} and \citet[Theorem~2]{CMS07} apply only to the case of random $v$, 
and therefore are not directly comparable to the results in \cite{GM11, KH16a} and \cite[Theorem~3]{CMS07}, nor to our result in \autoref{thmProj}, all of which consider non-random directions $v$. 
Most importantly, however, none of the existing works for non-random $v$ provides a justification for the case in which the direction $v$ is less smooth than $X$, 
a feature captured in our framework by Condition~\ref{condProjScale}.

\section{Extra simulation results} \label{appExtraSim}

\begin{table}[htbp]
	\centering
	\caption{
		Average computation times (seconds) of proposed versus benchmark tests over 100 replications.
		Our function \texttt{FPCR} simultaneously computes both the \(L^{2}\) and supremum norms in \eqref{eqStatTest},
		using candidate truncation levels $\jj = \{1, \dots, 49\}$, 
		and provides the p-values for all threshold levels \(\rho \in \{0.75, 0.85, 0.95\}\).
	}
	\label{tb_computing}
	\renewcommand{\arraystretch}{1.2}
	%	\resizebox{0.95\linewidth}{!} 
	\medskip
	{
		\begin{tabular}{c|ccccc}
			\hline\hline
			& FPCR    & HMV13 & Lei14  & SC15 & LL25   \\ \hline
			$n=50$ & 175.328 & 1.294 & 1.130  & 11.123 & 234.46 \\
			$n=200$ & 344.008 & 4.374 & 2.086 & 407.483 & 1225.114\\
			\hline\hline
		\end{tabular}
	}
\end{table}

\begin{figure}[htbp]
	\centering
	\includegraphics[width=0.99\linewidth]{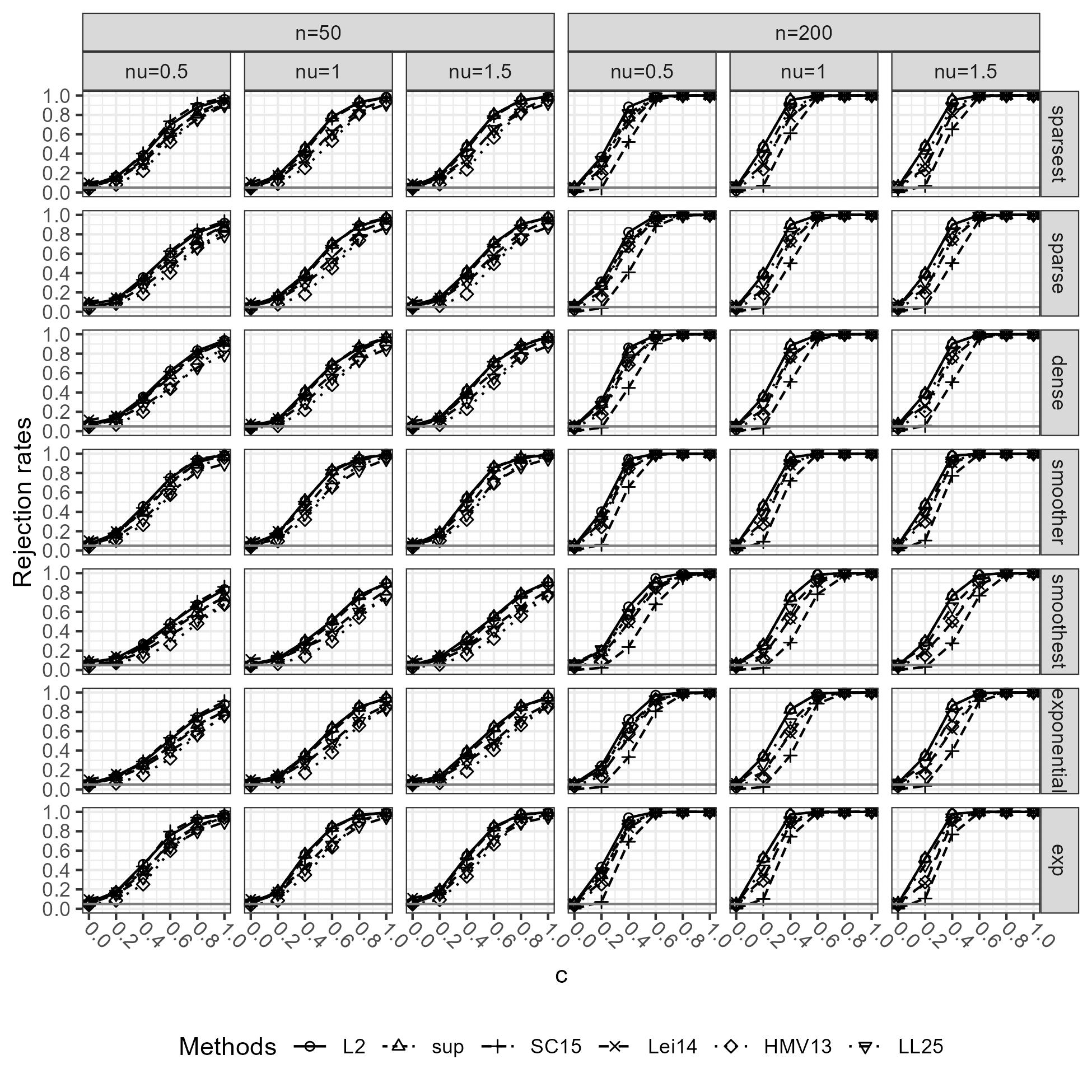}
	\caption{
		Extra results of the empirical rejection rates of the bootstrap tests using statistics  $S_{\mathrm{sq},J}$ and $S_{\mathrm{sup},J}$ in \eqref{eqStatTest} when $J$ is chosen by the FVE in \eqref{eqFVE} with $\rho=0.75$ and the benchmark tests.
	}
	\label{fig_test1sofr_appendix}
\end{figure}

\begin{figure}[htbp]
	\centering
	\includegraphics[width=0.99\linewidth]{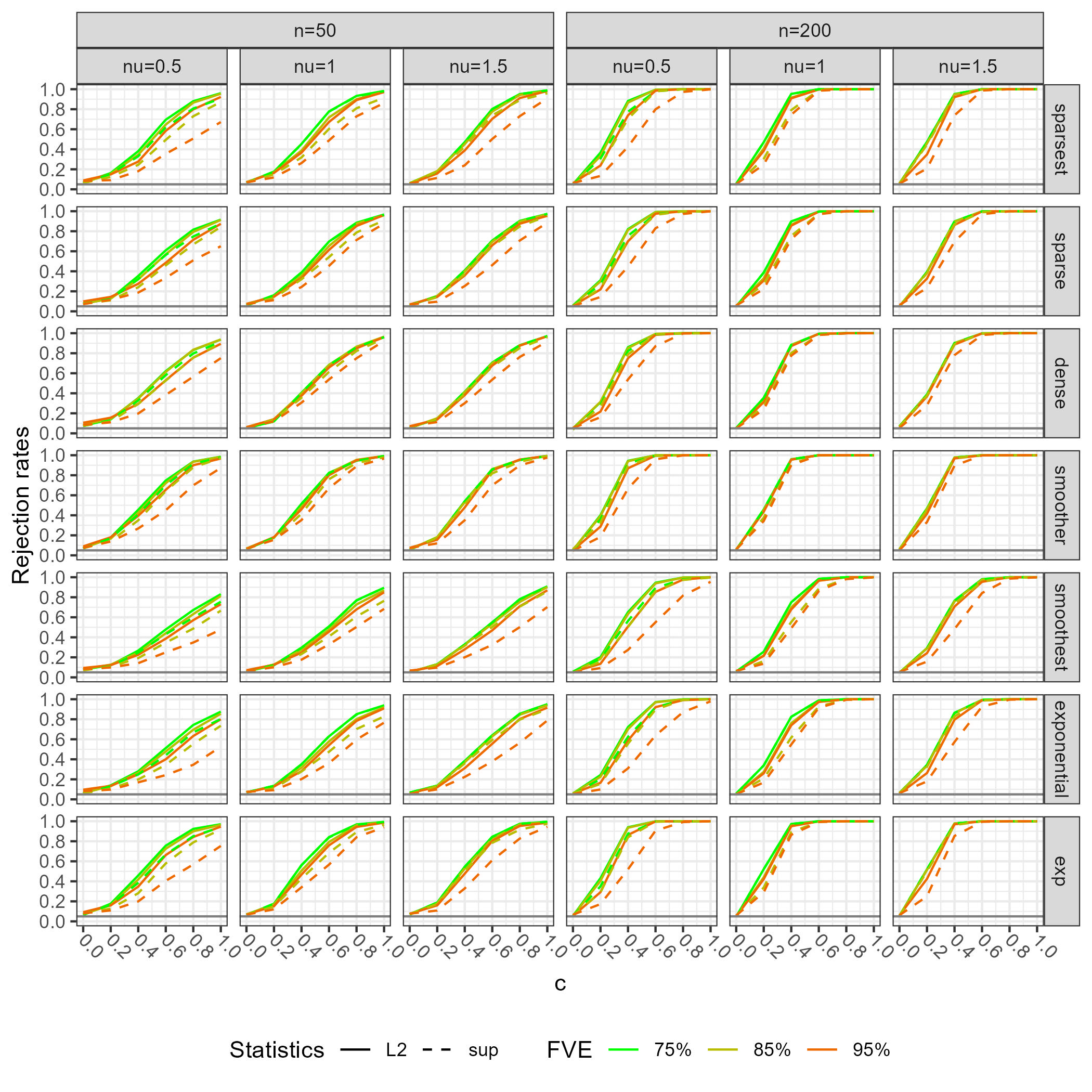}
	\caption{
		Extra results of the empirical rejection rates of the bootstrap tests using statistics  $S_{\mathrm{sq},J}$ and $S_{\mathrm{sup},J}$ in \eqref{eqStatTest} using various threshold values $\rho \in \{0.75, 0.85, 0.95\}$ for the FVE criterion in \eqref{eqFVE}. 
	}
	\label{fig_test1sofr_FVE}
\end{figure}

\end{appendix}

%%%%%%%%%%%%%%%%%%%%%%%%%%%%%%%%%%%%%%%%%%%%%%
%% Support information, if any,             %%
%% should be provided in the                %%
%% Acknowledgements section.                %%
%%%%%%%%%%%%%%%%%%%%%%%%%%%%%%%%%%%%%%%%%%%%%%
\section*{Acknowledgments}
%The author gratefully acknowledges a reviewer’s thoughtful feedback, particularly the suggestions to broaden the simulation scenarios and to expand the discussion on selecting truncation levels. 
The author thanks the authors of \cite{LL25} for kindly sharing their code,
which significantly helps implementing the benchmark tests.
%This work was supported in part by startup funds from Kent State University.
Part of this work was carried out while the author was at Kent State University. 
\clearpage
\bibliographystyle{dcu}
\bibliography{test1sofr}

@book{Joll86,
	title        = {Principal Component Analysis},
	author       = {Jolliffe, Ian T.},
	edition      = {1},
	series       = {Springer Series in Statistics},
	publisher    = {Springer},
	address      = {New York, NY},
	year         = {1986}
}

@book{DS88,
  title={Linear Operators, Part 1: General Theory},
  author={Dunford, Nelson and Schwartz, Jacob T},
  volume={10},
  year={1988},
  publisher={John Wiley \& Sons}
}

@book{SLL99,
	title={Mathematical Handbook of Formulas and Tables},
  	author={Spiegel, Murray R and Lipschutz, Seymour and Liu, John},
	volume={65},
	year={1999},
	publisher={McGraw Hill}
}

@book{vaart98,
	title={Asymptotic Statistics},
	author={Van der Vaart, Aad W},
	year={1998},
	publisher={Cambridge University Press}
}

@book{Bosq00,
	title={Linear Processes in Function Spaces: Theory and Applications},
	author={Bosq, Denis},
	volume={149},
	year={2000},
	publisher={Springer Science \& Business Media}
}

@book{Evans10,
	title={Partial Differential Equations},
	author={Evans, Lawrence C},
	volume={19},
	year={2010},
	publisher={American Mathematical Society}
}

@book{HK12,
	title={Inference for Functional Data with Applications},
	author={Horv{\'a}th, Lajos and Kokoszka, Piotr},
	volume={200},
	year={2012},
	publisher={Springer Science \& Business Media}
}

@book{HE15,
	title={Theoretical Foundations of Functional Data Analysis, with an Introduction to Linear Operators},
	author={Hsing, Tailen and Eubank, Randall},
	volume={997},
	year={2015},
	publisher={John Wiley \& Sons}
}

@book{KR17,
	title={Introduction to Functional Data Analysis},
	author={Kokoszka, Piotr and Reimherr, Matthew},
	year={2017},
	publisher={Chapman and Hall/CRC}
}

@book{CGLC24,
	title={Functional Data Analysis with R},
	author={Crainiceanu, Ciprian M and Goldsmith, Jeff and Leroux, Andrew and Cui, Erjia},
	year={2024},
	publisher={CRC Press}
}

@article{Mallows72,
	title={A note on asymptotic joint normality},
	author={Mallows, Colin L},
	journal={Annals of Mathematical Statistics},
	pages={508--515},
	year={1972},
	publisher={JSTOR}
}

@article{free81,
	title={Bootstrapping regression models},
	author={Freedman, David A},
	journal={Annals of Statistics},
	volume={9},
	number={6},
	pages={1218--1228},
	year={1981},
	publisher={Institute of Mathematical Statistics}
}

@article{KF92,
	title={Bootstrapping stationary autoregressive moving-average models},
	author={Kreiss, Jens-Peter and Franke, J{\"u}rgen},
	journal={Journal of Time Series Analysis},
	volume={13},
	number={4},
	pages={297--317},
	year={1992},
	publisher={Wiley Online Library}
}

@article{Ball93,
	title={{The reverse isoperimetric problem for Gaussian measure}},
	author={Ball, Keith},
	journal={Discrete \& Computational Geometry},
	volume={10},
	number={4},
	pages={411--420},
	year={1993},
	publisher={Springer-Verlag Berlin, Heidelberg}
}

@article{CFS99,
  title={Functional linear model},
  author={Cardot, Herv{\'e} and Ferraty, Fr{\'e}d{\'e}ric and Sarda, Pascal},
  journal={Statistics \& Probability Letters},
  volume={45},
  number={1},
  pages={11--22},
  year={1999},
  publisher={Elsevier}
}

@article{KF00,
	title={Asymptotics for lasso-type estimators},
	author={Knight, Keith and Fu, Wenjiang},
	journal={Annals of Statistics},
	pages={1356--1378},
	year={2000},
	publisher={JSTOR}
}

@article{CFMS03,
	title={Testing hypotheses in the functional linear model},
	author={Cardot, Herv{\'e} and Ferraty, Fr{\'e}d{\'e}ric and Mas, Andr{\'e} and Sarda, Pascal},
	journal={Scandinavian Journal of Statistics},
	volume={30},
	number={1},
	pages={241--255},
	year={2003},
	publisher={Wiley Online Library}
}

@article{YMW05b,
	title={Functional linear regression analysis for longitudinal data},
	author={Yao, Fang and M{\"u}ller, Hans-Georg and Wang, Jane-Ling},
	journal={Annals of Statistics},
	pages={2873--2903},
	year={2005}
}

@article{CH06,
  title={Prediction in functional linear regression},
  author={Cai, T Tony and Hall, Peter},
  journal={Annals of Statistics},
  volume={34},
  number={5},
  pages={2159--2179},
  year={2006},
  publisher={Institute of Mathematical Statistics}
}

@article{HH07,
  title={Methodology and convergence rates for functional linear regression},
  author={Hall, Peter and Horowitz, Joel L},
  journal={Annals of Statistics},
  volume={35},
  number={1},
  pages={70--91},
  year={2007},
  publisher={Institute of Mathematical Statistics}
}

@article{CMS07,
	title={{CLT} in functional linear regression models},
	author={Cardot, Herv{\'e} and Mas, Andr{\'e} and Sarda, Pascal},
	journal={Probability Theory and Related Fields},
	volume={138},
	number={3-4},
	pages={325--361},
	year={2007},
	publisher={Springer}
}

@article{KMSZ08,
	title={Testing for lack of dependence in the functional linear model},
	author={Kokoszka, Piotr and Maslova, Inga and Sojka, Jan and Zhu, Lie},
	journal={Canadian Journal of Statistics},
	volume={36},
	number={2},
	pages={207--222},
	year={2008},
	publisher={Wiley Online Library}
}

@article{MY08,
	title={Functional additive models},
	author={M{\"u}ller, Hans-Georg and Yao, Fang},
	journal={Journal of the American Statistical Association},
	volume={103},
	number={484},
	pages={1534--1544},
	year={2008},
	publisher={Taylor \& Francis}
}

@article{YC10,
	title={A reproducing kernel {Hilbert}
         space approach to functional linear regression},
	author={Yuan, Ming and Cai, T Tony},
	journal={Annals of Statistics},
	volume={38},
	number={6},
	pages={3412-3444},
	year={2010},
	publisher={Institute of Mathematical Statistics}
}

@article{CL11,
	title={Bootstrapping lasso estimators},
	author={Chatterjee, Arindam and Lahiri, Soumendra Nath},
	journal={Journal of the American Statistical Association},
	volume={106},
	number={494},
	pages={608--625},
	year={2011},
	publisher={Taylor \& Francis}
}

@article{Ross11,
	title={{Fundamentals of Stein’s method}},
	author={Ross, Nathan},
	journal={Probability Surveys},
	volume={8},
	pages={210--293},
	year={2011}
}

@article{GM11,
  title={Bootstrap in functional linear regression},
  author={Gonz{\'a}lez-Manteiga, Wenceslao and Mart{\'\i}nez-Calvo, Adela},
  journal={Journal of Statistical Planning and Inference},
  volume={141},
  number={1},
  pages={453--461},
  year={2011},
  publisher={Elsevier}
}

@article{SL12,
	title={On prediction rate in partial functional linear regression},
	author={Shin, Hyejin and Lee, Myung Hee},
	journal={Journal of Multivariate Analysis},
	volume={103},
	number={1},
	pages={93--106},
	year={2012},
	publisher={Elsevier}
}

@article{DPZ12,
	author    = {Winston Wei Dou and David Pollard and Harrison H. Zhou},
	title     = {Estimation in functional regression for general exponential families},
	journal   = {Annals of Statistics},
	volume    = {40},
	number    = {5},
	pages     = {2421--2451},
	year      = {2012}
}

@article{HKR13,
	title={Estimation of the mean of functional time series and a two-sample problem},
	author={Horv{\'a}th, Lajos and Kokoszka, Piotr and Reeder, Ron},
	journal={Journal of the Royal Statistical Society Series B: Statistical Methodology},
	volume={75},
	number={1},
	pages={103--122},
	year={2013},
	publisher={Oxford University Press}
}

@article{UF13,
	title={{Applications of functional data analysis: A systematic review}},
	author={Ullah, Shahid and Finch, Caroline F},
	journal={BMC medical research methodology},
	volume={13},
	number={1},
	pages={43},
	year={2013},
	publisher={Springer}
}

@article{SGS13,
	title={An introduction with medical applications to functional data analysis},
	author={S{\o}rensen, Helle and Goldsmith, Jeff and Sangalli, Laura M},
	journal={Statistics in Medicine},
	volume={32},
	number={30},
	pages={5222--5240},
	year={2013},
	publisher={Wiley Online Library}
}

@article{HMV13,
	title={Minimax adaptive tests for the functional linear model},
	author={Hilgert, Nadine and Mas, Andr{\'e} and Verzelen, Nicolas},
	journal={Annals of Statistics},
	volume={41},
	number={2},
	pages={838--869},
	year={2013},
	publisher={Institute of Mathematical Statistics}
}

@article{CM13,
	title={Asymptotics of prediction in functional linear regression with functional outputs},
	author={Crambes, Christophe and Mas, Andr{\'e}},
	journal={Bernoulli},
	volume={19},
	number={5B},
	pages={2627--2651},
	year={2013},
	publisher={Bernoulli Society for Mathematical Statistics and Probability}
}

@article{lei14,
	title={Adaptive global testing for functional linear models},
	author={Lei, Jing},
	journal={Journal of the American Statistical Association},
	volume={109},
	number={506},
	pages={624--634},
	year={2014},
	publisher={Taylor \& Francis}
}

@article{Cuevas14,
	title={A partial overview of the theory of statistics with functional data},
	author={Cuevas, Antonio},
	journal={Journal of Statistical Planning and Inference},
	volume={147},
	pages={1--23},
	year={2014},
	publisher={Elsevier}
}

@article{SC15,
	title={Nonparametric inference in generalized functional linear models},
	author={Shang, Zuofeng and Cheng, Guang},
	journal={Annals of Statistics},
	pages={1742--1773},
	year={2015},
	publisher={JSTOR}
}

@article{KXYZ16,
	title={Partially functional linear regression in high dimensions},
	author={Kong, Dehan and Xue, Kaijie and Yao, Fang and Zhang, Hao H},
	journal={Biometrika},
	volume={103},
	number={1},
	pages={147--159},
	year={2016},
	publisher={Oxford University Press}
}

@article{KH16a,
  title={On asymptotic distribution of prediction in functional linear regression},
  author={Khademnoe, Omid and Hosseini-Nasab, S Mohammad E},
  journal={Statistics},
  volume={50},
  number={5},
  pages={974--990},
  year={2016a},
  publisher={Taylor \& Francis}
}

@article{KSM16,
	title={Classical testing in functional linear models},
	author={Kong, Dehan and Staicu, Ana-Maria and Maity, Arnab},
	journal={Journal of Nonparametric Statistics},
	volume={28},
	number={4},
	pages={813--838},
	year={2016},
	publisher={Taylor \& Francis}
}

@article{PCTM18,
	title={Singular additive models for function to function regression},
	author={Park, Byeong U and Chen, Chun-Jui and Tao, Wenwen and M{\"u}ller, Hans-Georg},
	journal={Statistica Sinica},
	volume={28},
	number={4},
	pages={2497--2520},
	year={2018},
	publisher={JSTOR}
}

@article{CR18,
	title={A geometric approach to confidence regions and bands for functional parameters},
	author={Choi, Hyunphil and Reimherr, Matthew},
	journal={Journal of the Royal Statistical Society Series B: Statistical Methodology},
	volume={80},
	number={1},
	pages={239--260},
	year={2018},
	publisher={Oxford University Press}
}

@article{DBZ17,
	title={High-dimensional simultaneous inference with the bootstrap},
	author={Dezeure, Ruben and B{\"u}hlmann, Peter and Zhang, Cun-Hui},
	journal={Test},
	volume={26},
	number={4},
	pages={685--719},
	year={2017},
	publisher={Springer}
}

@article{Eck18,
	title={Bootstrapping for multivariate linear regression models},
	author={Eck, Daniel J},
	journal={Statistics \& Probability Letters},
	volume={134},
	pages={141--149},
	year={2018},
	publisher={Elsevier}
}

@article{LY19,
	title={Intrinsic Riemannian functional data analysis},
	author={Lin, Zhenhua and Yao, Fang},
	journal={Annals of Statistics},
	volume={47},
	number={6},
	pages={3533--3577},
	year={2019}
}

@article{Raic19,
	author = {Martin Raič},
	title = {{A multivariate Berry–Esseen theorem with explicit constants}},
	volume = {25},
	journal = {Bernoulli},
	number = {4A},
	publisher = {Bernoulli Society for Mathematical Statistics and Probability},
	pages = {2824 -- 2853},
	year = {2019}
}

@article{LZ20,
	title={Inference for generalized partial functional linear regression},
	author={Li, Ting and Zhu, Zhongyi},
	journal={Statistica Sinica},
	volume={30},
	number={3},
	pages={1379--1397},
	year={2020},
	publisher={JSTOR}
}

@article{Bonis20,
	title={{Stein’s method for normal approximation in Wasserstein distances with application to the multivariate central limit theorem}},
	author={Bonis, Thomas},
	journal={Probability Theory and Related Fields},
	volume={178},
	number={3},
	pages={827--860},
	year={2020},
	publisher={Springer}
}

@article{Dai22,
	title={{Statistical inference on the Hilbert sphere with application to random densities}},
	author={Dai, Xiongtao},
	journal={Electronic Journal of Statistics},
	volume={16},
	number={1},
	pages={700--736},
	year={2022},
	publisher={The Institute of Mathematical Statistics and the Bernoulli Society}
}

@article{LL25,
	title={Hypothesis testing for functional linear models via bootstrapping},
	author={Lin, Yinan and Lin, Zhenhua},
	journal={Bernoulli},
	volume={31},
	number={4},
	pages={3309--3330},
	year={2025},
	publisher={Bernoulli Society for Mathematical Statistics and Probability}
}

@article{Yang25,
	title={{Continuity of Gaussian extreme distributions}},
	author={Yang, Lijian},
	journal={Statistics \& Probability Letters},
	volume={216},
	pages={110274},
	year={2025},
	publisher={Elsevier}
}

@article{KDD22,
	title={Statistical inference for the slope parameter in functional linear regression},
	author={Kutta, Tim and Dierickx, Gauthier and Dette, Holger},
	journal={Electronic Journal of Statistics},
	volume={16},
	number={2},
	pages={5980--6042},
	year={2022},
	publisher={The Institute of Mathematical Statistics and the Bernoulli Society}
}

@article{CLM23,
	title={Wasserstein regression},
	author={Chen, Yaqing and Lin, Zhenhua and M{\"u}ller, Hans-Georg},
	journal={Journal of the American Statistical Association},
	volume={118},
	number={542},
	pages={869--882},
	year={2023},
	publisher={Taylor \& Francis}
}

@article{CKR23,
  	title={Inference for projection parameters in linear regression: beyond d=o(n\textsuperscript{1/2})},
	author={Chang, Woonyoung and Kuchibhotla, Arun Kumar and Rinaldo, Alessandro},
	journal={arXiv preprint arXiv:2307.00795},
	year={2023}
}

@article{LLM23,
	title={{High-dimensional MANOVA via bootstrapping and its application to functional and sparse count data}},
	author={Lin, Zhenhua and Lopes, Miles E and M{\"u}ller, Hans-Georg},
	journal={Journal of the American Statistical Association},
	volume={118},
	number={541},
	pages={177--191},
	year={2023},
	publisher={Taylor \& Francis}
}

@article{ZYZ23,
	title={Functional linear regression for discretely observed data: from ideal to reality},
	author={Zhou, Hang and Yao, Fang and Zhang, Huiming},
	journal={Biometrika},
	volume={110},
	number={2},
	pages={381--393},
	year={2023},
	publisher={Oxford University Press}
}

@article{KH24,
	title={On the validity of the bootstrap hypothesis testing in functional linear regression},
	author={Khademnoe, Omid and Hosseini-Nasab, S Mohammad E},
	journal={Statistical Papers},
	volume={65},
	number={4},
	pages={2361--2396},
	year={2024},
	publisher={Springer}
}

@article{DT24,
	title={Statistical inference for function-on-function linear regression},
	author={Dette, Holger and Tang, Jiajun},
	journal={Bernoulli},
	volume={30},
	number={1},
	pages={304--331},
	year={2024},
	publisher={Bernoulli Society for Mathematical Statistics and Probability}
}

@article{Warm24,
	title={{Over 30 years of using functional data analysis in human movement. What do we know, and is there more for sports biomechanics to learn?}},
	author={Warmenhoven, John},
	journal={Sports Biomechanics},
	pages={1--32},
	year={2024},
	publisher={Taylor \& Francis}
}

@article{ST25,
	title={{Operator learning: A statistical perspective}},
	author={Subedi, Unique and Tewari, Ambuj},
	journal={Annual Review of Statistics and Its Application},
	volume={13},
	year={2025},
	publisher={Annual Reviews}
}

@article{CCYZ25,
	title={High-dimensional multiresponse partially functional linear regression},
	author={Cai, Xiong and Cao, Jiguo and Yan, Xingyu and Zhao, Peng},
	journal={Statistics in Medicine},
	volume={44},
	number={13-14},
	pages={e70140},
	year={2025},
	publisher={Wiley Online Library}
}

@article{CP25,
	title={{High-dimensional Hilbert--Schmidt linear regression with Hilbert manifold variables}},
	author={Choi, Changwon and Park, Byeong U},
	journal={Annals of Statistics},
	volume={53},
	number={6},
	pages={2673--2701},
	year={2025},
	publisher={Institute of Mathematical Statistics}
}

@article{PRRP26,
	title={{PCA for point processes}},
	author={Picard, Franck and Rivoirard, Vincent and Roche, Angelina and Panaretos, Victor M},
	journal={Annals of Statistics},
	volume={54},
	number={2},
	pages={929--956},
	year={2026},
	publisher={Institute of Mathematical Statistics}
}

@article{SP26,
	title={{Statistical inference for Bures-Wasserstein flows}},
	author={Santoro, Leonardo V and Panaretos, Victor M},
	journal={To appear in Bernoulli},
	year={2026+}
}

@article{YDN23RB,
	author = {Hyemin Yeon and Xiongtao Dai and Daniel J. Nordman},
	title = {{Bootstrap inference in functional linear regression models with scalar response}},
	volume = {29},
	journal = {Bernoulli},
	number = {4},
	publisher = {Bernoulli Society for Mathematical Statistics and Probability},
	pages = {2599 -- 2626},
	year = {2023}
}

@article{YDN25WB,
	title={Wild bootstrap for mean response inference in functional linear regression models},
	author={Yeon, Hyemin and Dai, Xiongtao and Nordman, Daniel J},
	year={2026+},
	journal={Under revision}
}

@article{Yeon26,
	title = {Inference for function-on-function regression: central limit theorem and residual bootstrap},
	author = {Yeon, Hyemin},
	journal = {To appear in Statistics and Its Interface},
	year = {2026+}
}

@article{KKD24,
	title={Generalized linear models with spatial dependence and a functional covariate},
	author={Kim, Sooran and Kaiser, Mark S and Dai, Xiongtao},
	journal={arXiv preprint arXiv:2402.13472},
	year={2024}
}

@article{WCM16,
	title={Functional data analysis},
	author={Wang, Jane-Ling and Chiou, Jeng-Min 
                        and M{\"u}ller, Hans-Georg},
	journal={Annual Review of Statistics and its application},
	volume={3},
	number={1},
	pages={257--295},
	year={2016},
	publisher={Annual Reviews}
}

@article{Morr15,
	title={Functional regression},
	author={Morris, Jeffrey S},
	journal={Annual Review of Statistics and Its Application},
	volume={2},
	number={1},
	pages={321--359},
	year={2015},
	publisher={Annual Reviews}
}

@article{GRLG24,
	title={{Functional data analysis: An introduction and recent developments}},
	author={Gertheiss, Jan and R{\"u}gamer, David and Liew, Bernard XW and Greven, Sonja},
	journal={Biometrical Journal},
	volume={66},
	number={7},
	pages={e202300363},
	year={2024},
	publisher={Wiley Online Library}
}

@article{AD25,
	title={Functional data analysis for wearable sensor data: a systematic review},
	author={Acar-Denizli, Nihan and Delicado, Pedro},
	journal={AStA Advances in Statistical Analysis},
	pages={1--41},
	year={2025},
	publisher={Springer}
}

@article{OOO25,
	title={Functional data analysis applications in medicine: a systematic review},
	author={Orozco, Nicolas and Ortiz, Santiago and Ospina-Tasc{\'o}n, Gustavo A},
	journal={Wiley Interdisciplinary Reviews: Computational Statistics},
	volume={17},
	number={2},
	pages={e70026},
	year={2025},
	publisher={Wiley Online Library}
}

@article{LS26,
	title={{Spectrum-aware debiasing: A modern inference framework with applications to principal components regression}},
	author={Li, Yufan and Sur, Pragya},
	journal={Annals of Statistics},
	volume={54},
	number={2},
	pages={745--770},
	year={2026},
	publisher={Institute of Mathematical Statistics}
}

\end{document}